\documentclass[11pt,a2paper]{article}

\usepackage[T1]{fontenc} 
\usepackage{lmodern}

\usepackage[utf8]{inputenc}
\usepackage{amsfonts} 
\usepackage{dsfont}
\usepackage{amsmath}
\usepackage{amsmath,latexsym}
\usepackage{amsmath,amssymb}
\usepackage{amsmath,amsthm}
\usepackage[english]{babel}
\usepackage{moreverb}
\usepackage[]{graphics}
\usepackage[]{latexsym} 
\usepackage{amscd}
\usepackage{color}
\usepackage{lineno}
\usepackage{ae}
\usepackage{enumerate}

 \usepackage[figuresright]{rotating}
\usepackage{bigints}
 \graphicspath{{FIGURES/}}
 \makeatletter
\newcommand{\xrightharpoonup}[2][]{\ext@arrow 0359\rightharpoonupfill@{#1}{#2}}
\makeatother
\usepackage{hyperref}
\hypersetup{
	colorlinks=true,
	linkcolor=blue,
	filecolor=magenta,      
	urlcolor=cyan,
}

\usepackage{graphics}
\usepackage{graphicx} 
 
\newcommand{\ds}{\displaystyle}
\newcommand{\ve}{\varepsilon}

\def\R{\mathbb{R}}

\def\P{\mathbb{P}}
\def\E{\mathbb{E}}

\def\N{{\rm I\hspace{-0.50ex}N} }
\def\N{\mathbb{N}}

\newtheorem{lem}{Lemma}[section]
\newtheorem{thm}[lem]{Theorem}

\newtheorem{defi}[lem]{Definition}
\newtheorem{prop}[lem]{Proposition}

\newtheorem{rmq}[lem]{Remark}

 \title{\bf $L^s$-rate optimality of dilated$/$contracted $L^r$-optimal and greedy quantization sequences }
\author{\textsc{Rancy El Nmeir} \thanks{Sorbonne Universit\'e, Laboratoire de Probabilit\'e, Statistique et Mod\'elisation, Campus Pierre et Marie Curie, case 158, 4, pl. Jussieu, F-75252 Paris Cedex 5, France.} \thanks{Universit\'e Saint-Joseph de Beyrouth, Laboratoire de Math\'ematiques et Applications, Unit\'e de recherche Math\'ematiques et mod\'elisation, B.P. 11-514 Riad El Solh Beyrouth 1107
2050, Liban.}}

\begin{document}
	\date{}
	\maketitle
	
	\begin{abstract}
	We investigate some $L^s$-rate optimality properties of dilated/contracted $L^r$-optimal quantizers and $L^r$-greedy quantization sequences $(\alpha^n)_{n \geq 1}$ of a random variable $X$. We establish, for different values of $s$, $L^s$-rate optimality results for $L^r$-optimally dilated/contracted greedy quantization sequences  $(\alpha^n_{\theta,\mu})_{n \geq 1}$ defined by $\alpha^n_{\theta,\mu}=\{\mu+\theta (\alpha_i-\mu), \alpha_i \in \alpha^{(n)}\}$. We lead a specific study for $L^r$-optimal greedy quantization sequences of radial density distributions and show that they are $L^s$-rate optimal for $s \in (r,r+d)$ under some moment assumption. Based on the results established in $\cite{Sagna08}$ for $L^r$-optimal quantizers, we show, for a larger class of distributions, that the dilatation $(\alpha^n_{\theta,\mu})_{n \geq 1}$ of an $L^r$-optimal quantizer is $L^s$-rate optimal for $s < r+d$. We show, for various probability distributions, that there exists a parameter $\theta^*$ for which the dilated quantization sequence satisfy the so-called {\em $L^s$-empirical measure} theorem and present an application of this approach to numerical integration.
	\end{abstract}
	\bigskip
\label{dilatationgreedy}
\paragraph{Abstract}
	We investigate some $L^s$-rate optimality properties of dilated/contracted $L^r$-optimal quantizers and $L^r$-greedy quantization sequences $(\alpha^n)_{n \geq 1}$ of a random variable $X$. We establish, for different values of $s$, $L^s$-rate optimality results for $L^r$-optimally dilated/contracted greedy quantization sequences  $(\alpha^n_{\theta,\mu})_{n \geq 1}$ defined by $\alpha^n_{\theta,\mu}=\{\mu+\theta (\alpha_i-\mu), \alpha_i \in \alpha^{(n)}\}$. We lead a specific study for $L^r$-optimal greedy quantization sequences of radial density distributions and show that they are $L^s$-rate optimal for $s \in (r,r+d)$ under some moment assumption. Based on the results established in $\cite{Sagna08}$ for $L^r$-optimal quantizers, we show, for a larger class of distributions, that the dilatation $(\alpha^n_{\theta,\mu})_{n \geq 1}$ of an $L^r$-optimal quantizer is $L^s$-rate optimal for $s < r+d$. We show, for various probability distributions, that there exists a parameter $\theta^*$ for which the dilated quantization sequence satisfy the so-called {\em $L^s$-empirical measure} theorem and present an application of this approach to numerical integration.
\section{Introduction}
The aim of this paper is, on the one hand, to extend some ``robustness'' results of optimal quantizers to a much wider class of   distributions and, on the other hand, to establish similar results for greedy quantization sequences introduced in $\cite{LuPa15}$ and developed in \cite{papiergreedy}. Let $L^r_{\mathbb{R}^d}(\mathbb{P})$ (or simply  $L^r (\mathbb{P})$), $r \in (0, +\infty)$, denote the set of $d$-dimensional random vectors $X$ defined on the probability space $(\Omega,\mathcal{A},\mathbb{P})$ with distribution $P=\P_X$ and such that $\E|X|^r<+\infty$ (for any norm $|\cdot|$ on $\R^d$). Optimal vector quantization consists in finding the best approximation of a multidimensional random vector $X$ by a random variable $Y$ taking at most a finite number $n$ of values. Consider $\Gamma=\{x_1,\ldots,x_n\}$ a $d$-dimensional grid of size $n$. The principle is to approximate $X$ by $\pi_{\Gamma}(X)$ where $\pi_{\Gamma}:\R^d \rightarrow \Gamma$ is a nearest neighbor projection defined by \[\pi_{\Gamma} (\xi) = \sum_{i=1}^{n} x_i \mathds{1}_{W_i(\Gamma)}(\xi)\] 
where $\big(W_i(\Gamma)\big)_{1\leq i\leq n}$ is a so-called {\em Vorono\"i partition} of $\R^d$ induced by $\Gamma$ i.e. a Borel partition satisfying
\begin{equation}
	\label{Voronoicells}
	W_{i}(\Gamma)  \subset \big\{\xi \in \mathbb{R}^d : |\xi-x_i|\leq  \min_{j\neq i} |\xi-x_j|\big\}, \qquad i =1,\ldots,n.
	\end{equation}Then, 
	\begin{equation} 
	\widehat{X}^{\Gamma} =\pi_{\Gamma}(X) := \sum_{i=1}^{n}x_i \mathds{1}_{W_{i}(\Gamma)}(X)
	\end{equation}
	is called the {\em Vorono\"i quantization} of $X$.
	The $L^r$-quantization error induced when replacing $X$ by its quantization $\widehat X^{\Gamma}$ is naturally defined by
	\begin{equation} 
	\label{quanterror}
	e_r(\Gamma,X)  =  \|X-\pi_{\Gamma}(X)\|_r = \|X-\widehat{X}^{\Gamma}\|_r = \left \|\min_{1\leq i \leq n} |X - x_i|\right \|_r
	\end{equation}
	where $\|.\|_r$ denotes the $L^r(\P)$-norm (or quasi-norm if $0<r<1$). Consequently, the optimal quantization problem at {\em level} $n$ boils down to finding the grid $\Gamma^n$ of size $n$ that minimizes this error, i.e.
	\begin{equation}
	e_{r,n}(X) \; = \; \inf_{\Gamma, \; {\rm card}(\Gamma) \leq n} e_r(\Gamma,X).
	\end{equation}
	where ${\rm card}(\Gamma)$ denotes the cardinality of $\Gamma$. The existence of a solution to this problem and the convergence of $e_{r,n}(X)$ to $0$ at an $\mathcal{O}(n^{-\frac 1d})$-rate of convergence when the level (or size) $n$ goes to $+\infty$ have been shown (see \cite{GraLu00,Pages15,Pages18} for example). The  convergence to $0$ of such an error induced by a  sequence $(\Gamma^n)_{n\geq 1}$ of $L^r$-optimal quantizers  of (the distribution of) $X$ is an easy consequence of the separability of $\R^d$. Its rate of convergence to $0$ is a much more challenging problem that has been solved in several steps over between $1950$'s and the early $2000$'s and the main results in their final form are summed up in Section \ref{tools}.\\

However, numerical implementation of multidimensional $L^r$-optimal quantizers requires to optimize grids of size $n \times d$ which becomes computationally too costly
when $n$ or $d$ increase. So, a greedy version of optimal vector quantization (which is easier to handle) has been introduced in \cite{LuPa15} as a sub-optimal solution to the quantization problem. It consists in building a {\it sequence} of points $(a_n)_{n \geq 1}$ in $\R^d$ which is recursively $L^r$-optimized level by level, in the sense that it minimizes the $L^r$-quantization error at each iteration in a greedy way. This means that, having the first $n$ points $a^{(n)}=\{a_1,\ldots,a_n\}$ for $n\geq 1$, we add, at the $(n+1)$-th step, the point $a_{n+1}$ solution to
\begin{equation}
\label{dilat:greedydef}
\qquad a_{n+1} \in \mbox{argmin}_{\xi \in \R^d} \, e_r(a^{(n)} \cup \{\xi\}, X),
\end{equation}  
noting that $a^{(0)} = \varnothing$, so that {\em $a_1$ is simply an/the $L^r $-median of the distribution $P$ of $X$}. 
The sequence $(a_n)_{n \geq 1}$ is called an $L^r$-optimal greedy quantization sequence for $X$ or its distribution $P$. It is proved in $\cite{LuPa15}$ that the problem $(\ref{dilat:greedydef})$ admits, as soon as $X$ lies in $L_{\R^d}(\P)$,  a solution  $(a_n)_{n \geq 1}$ which may be not unique due to the dependence of greedy quantization on the symmetry of the distribution $P$. The corresponding $L^r$-quantization error $e_r(a^{(n)},X)$ is decreasing w.r.t $n$ and converges to $0$ when $n$ goes to $+\infty$. Greedy quantization sequences have an optimal convergence rate to $0$ compared to optimal quantizers, in the sense that the grids $\{a_1,\ldots,a_n\}$ are $L^r$-rate optimal, i.e. the corresponding quantization error converges with an $\mathcal{O}(n^{-\frac 1d})$-rate of convergence. This was established first in $\cite{LuPa15}$ for a rather wide family of absolutely continuous distribution using some maximal functions approximating the density $f$ of $P$. Then, it has been extended in $\cite{papiergreedy}$ to a much larger class of probability density functions where the authors relied on an exogenous auxiliary probability distribution $\nu$ on $(\R^d,\mathcal{B}or(\R^d))$ satisfying a certain control on balls, the result is recalled in Section \ref{tools}.\\

A very important field of applications is quantization-based numerical integration where we approximate an expectation $\E h(X)$ of a function $h$ on $\R^d$ by some cubature formulas. The error bounds induced by such numerical schemes always involve the $L^s$-quantization error induced by the approximation of $X$ by its (optimal or greedy) quantization 
usually with $s \geq r$. This problem also appears when we use optimal quantization as a space discretization scheme of ARCH models, namely the Euler scheme of a diffusion devised to solve stochastic control, optimal stopping or filtering problems (see \cite{PaPhPr03,PaPhPr04} for example) where,  in order to estimate the upper error bounds induced by such approximation schemes, one needs to evaluate $L^s$-quantization errors induced by $L^r$-optimal (or asymptotically optimal) quantizers for $s \geq r$.
So, one needs to see whether such quantizers sharing $L^r$-optimality properties preserve their performances in $L^s$, this is called the {\em distortion mismatch} problem and was deeply studied in $\cite{mismatch08}$ for sequences of optimal quantizers. As for greedy quantization sequences, it was first investigated in $\cite{LuPa15}$ and extended later in $\cite{papiergreedy}$ as already mentioned. \medskip\\

Another approach to this problem was considered in $\cite{Sagna08}$ where the author was interested in the fact that an appropriate dilatation or contraction of a (sequence of) $L^r$-optimal quantizer(s) $(\Gamma^n)_{n \geq 1}$ remains $L^s$-rate optimal. This study was also motivated by its application to the algorithms of designing $L^s$-optimal quantizers for $s\neq 2$. In fact, several stochastic procedures, like Lloyd's algorithm or the Competitive Learning Vector Quantization algorithm (CLVQ), are based on the stationarity property satisfied by optimal quadratic quantizers and designed for $s=2$. However, when $s >2$, these procedures become unstable and difficult and their convergence is very dependent on the initialization. So, in order to design $L^s$-optimal quantizers, $s>2$, one can use the $L^2$-dilated quantizers to initialize the algorithms and speed their convergence. \\

In this paper, based on the same motivations, we are interested in establishing $L^s$-rate optimality results of dilatations/contractions of $L^r$-optimal greedy quantization sequences. 
Moreover, we extend the original results established for $L^r$-optimal quantizers in $\cite{Sagna08}$ to a larger class of distributions taking advantage of new tools developed in $\cite{papiergreedy}$ to analyze quantization errors. These tools are based on auxiliary probability distributions with a certain property of control on balls. In other words, if $(\alpha^n)_{n \geq 1}$ is a sequence of $L^r$-optimal quantizers or an $L^r$-optimal greedy quantization sequence, then the sequence $(\alpha^n_{\theta,\mu})_{n \geq 1}$ defined, for every $\theta >0$ and $\mu \in \R^d$, by $\alpha_{\theta,\mu}^n=\{\mu+\theta(a_i-\mu),\; a_i \in \alpha^n\}$, is $L^s$-rate optimal for $s \neq r$.  A lower bound of the $L^s$-quantization error $e_s(\alpha^n_{\theta,\mu},P)$ was given in $\cite{Sagna08}$ for $L^r$-optimal quantizers and it also holds for greedy quantization sequences: If $P=f.\lambda_d$, then for every $\theta>0$, $\mu \in \R^d$ and $n \geq 1$,
\begin{equation}
\label{lowerbound}
\liminf _{n \rightarrow +\infty} n^{\frac 1d} e_s(\alpha^n_{\theta,\mu},P) \geq Q_{r,s}^{\text{Inf}}(P,\theta) \end{equation}
where \begin{align}
Q_{r,s}^{\text{Inf}}(P,\theta) &= \theta^{1+\frac ds} \widetilde{J}_{s,d} \left( \int_{\R^d} f^{\frac{d}{d+r}}d\lambda_d \right)^{\frac 1d} \left(\int_{\{f > 0\}} f_{\theta,\mu} f^{-\frac{s}{d+r}}d\lambda_d\right)^{\frac1s}\nonumber \\
&=\theta \widetilde{J}_{s,d} \left( \int_{\R^d} f^{\frac{d}{d+r}}d\lambda_d \right)^{\frac 1d} \left(\int_{\{f > 0\}} f^{-\frac{s}{d+r}}dP_{\theta,\mu}\right)^{\frac1s}
\end{align}
where $\widetilde{J}_{s,d} = \ds \inf_{n \geq 1} n^{\frac{1}{d}} e_{s,n}(U([0,1]^d)) \in (0,+ \infty)$ is the constant given in Zador's Theorem (see (\ref{Zadorgeneral})) and $f_{\theta,\mu}$ denotes the function $f_{\theta,\mu}(x)=f(\mu+\theta(x-\mu))$. Likewise, if $X \sim P=f.\lambda_d$, then $P_{\theta,\mu}$ denotes the probability distribution of the random variable $\frac{X-\mu}{\theta}+\mu$ and $dP_{\theta,\mu}=\theta^d f_{\theta,\mu}.d\lambda_d$. Our goal is then to estimate upper bounds of this error. For the $L^r$-dilated/contracted greedy quantization sequences, we rely on auxiliary probability distributions satisfying a certain control criterion on balls and establish upper estimates depending on the values of $s$. We obtain Pierce type universal non-asymptotic results of $L^s$-rate optimality of a greedy quantization sequence $(\alpha^n_{\theta,\mu})_{n \geq 1}$ of a distribution $P$ having finite polynomial moments at any order. On another hand, we lead an interesting study for a particular class of distributions, the radial density probability distributions, showing that the corresponding $L^r$-greedy quantization sequences are $L^s$-rate optimal for $s\in (r,d+r)$ under some moment assumption on $P$ and we investigate a particular case, the Hyper-Cauchy distribution, where the distribution $P$ has finite polynomial moments up to a finite order. As for the $L^r$-dilated/contracted optimal quantizers, two results are already given in \cite{Sagna08}: one showing that an asymptotically $L^r$-optimal sequence of quantizers is $L^s$-rate optimal and another restricted to a sequence of (exactly) $L^r$-optimal quantizers and showing that it is $L^s$-rate optimal for $s \in (0,+\infty)$. In this paper, we change the approach and use auxiliary probability distributions satisfying a control criterion on balls to extend these results to a larger class of distributions for $L^r$-optimal quantizers. At this stage, one wonders if the $L^r$-dilated sequence satisfy the so-called $L^s$-empirical measure theorem or if there exists a particular set of parameters $(\theta^*,\mu^*)$ for which it is satisfied, leading to wonder whether the sequence is $L^s$-asymptotically optimal. This prompts us to consider several particular probability distributions and establish this study for each distribution. Finally, the application of this study to numerical integration, introduced in \cite{Sagna08}, is detailed and illustrated, by numerical examples, for optimal and greedy quantization.\\

This paper will be organized as follows: We start, in Section $\ref{tools}$, with some results and tools, mostly from $\cite{papiergreedy}$, that will be useful in the whole paper. In Section $\ref{greedydilat}$, we give upper bounds for dilated/contracted sequences of $L^r$-greedy quantization sequences of a distribution $P$ having finite polynomial moments at any order, investigate an example of a not so general case and lead a specific study for greedy quantization sequences of radial density distributions. Such error bounds are given for optimal quantizers in Section $\ref{optimaldilat}$. In Section $\ref{examplesdilat}$, we present several studies concerning the convergence of the empirical measure and the $L^s$-asymptotic optimality of the $L^r$-dilated/contracted sequence of particular probability distributions. Finally, Section \ref{application} is devoted to an application to numerical integration.
\section{Main tools}
\label{tools}
In this section, we present some useful results and inequalities which constitute essential tools needed to achieve desired results in the rest of the paper. Let $X$ be an $\R^d$-valued random variable with distribution $P$ such that $\E|X|^r<+\infty$ for $r>0$ and a norm $|\cdot|$ on $\R^d$. Let $(\Gamma^n)_{n\geq 0}$ be a sequence of $L^r$-optimal quantizers of $X$ and $(a_n)_{n \geq 0}$ be a corresponding greedy quantization sequence. We start by giving the result concerning the rate of convergence to $0$ of a sequence of $L^r$-optimal quantizers. The first part of the following theorem is an asymptotic result and the second part is universal non-asymptotic.
\begin{thm} 
		\label{Zadoretpierce}
		\noindent $(a)$ {\rm Zador's Theorem (see \cite{Zador82})} :  
		Let $X \in L_{\mathbb{R}^d}^{r+\eta}(\P)$, $\eta>0$, 
		with distribution $P$ such that 
		$dP(\xi)= \varphi(\xi) d \lambda_d(\xi)+ d\nu(\xi)$. 
		Then, 
		\begin{equation}
		\label{Zadorgeneral}
		\ds\lim_{n\rightarrow +\infty} n^{\frac{1}{d}} e_{r,n}(X) = Q_r(P)= \widetilde{J}_{r,d} \|\varphi\|^{\frac{1}{r}}_{L^{\frac{r}{r+d}}(\lambda_d)}
		\end{equation}
		where $\widetilde{J}_{r,d} = \ds \inf_{n \geq 1} n^{\frac{1}{d}} e_{r,n}(U([0,1]^d)) \in (0,+ \infty)$.  \\
		\noindent $(b)$ {\rm Extended Pierce's Lemma (see \cite{LuPa08,Pages18})}:
		Let $r,\eta >0 $. There exists a constant $\kappa_{d,r,\eta} \in (0,+\infty)$  such that, for any random vector $X: (\Omega, \mathcal{A}, \P)\rightarrow \R^d$,
		\begin{equation}
		\label{Piercegeneral}
		\forall n\geq 1,\quad e_{r,n}(X) \leq \kappa_{d,r,\eta} \sigma_{r+\eta} (X) n^{- \frac{1}{d}}
		\end{equation}
		where, for every $r \in (0, + \infty), \, \sigma_r(X) = \ds \inf_{a \in \mathbb{R}^d} \|X-a\|_r\leq +\infty$.
\end{thm}
\noindent Note that a sequence of $n$-quantizers $(\Gamma^n)_{n \geq 1}$ is said to be {\em asymptotically $L^r$-optimal} if $$\lim_{n} n^{\frac 1d} e_r(\Gamma^n,X)=Q_r(P)$$ and {\em $L^r$-rate optimal} if  \begin{equation}
\label{defrateoptimal}
\limsup_{n \rightarrow +\infty} n^{\frac1d} e_r(\Gamma^n,X)<+\infty \quad \mbox{or equivalently} \quad  \forall n\geq 1, \quad e_r(\Gamma^n,X) \leq C_1 n^{-\frac1d} 
\end{equation}
where $C_1$ is a constant not depending on $n$.\\

The $L^r$-rate optimality of greedy quantization sequences has been recently extended in $\cite{papiergreedy}$. The authors relied on auxiliary probability distributions $\nu$ on $(\R^d, \mathcal{B}(\R^d))$ satisfying the following control on balls, with respect to an $L^r$-median $a_1$ of $P$: Assume there exists $\varepsilon_0\!\in (0, 1]$ such that for every $\varepsilon \in (0,\varepsilon_0)$, there exists a Borel function $g_{\varepsilon}: \R^d \rightarrow [0,+\infty)$ such that, for every $x \in \mbox{supp}(P)$ and every $t \in [0,\varepsilon |x-a_1|]$, 
\begin{equation}
\label{criterenuintro}
  \nu(B(x,t)) \geq g_{\varepsilon}(x) V_d t^d
\end{equation}
where $V_d$ denotes the volume of the hyper unit ball. 
Of course, this condition is of interest only if the set $\{g_\varepsilon >0\}$ is sufficiently large with respect to $\{f>0\}$ (where $f$ is the density of $P$). 
\begin{thm}
	\label{greedyrateoptimal}
	{\rm (see \cite{papiergreedy})}
	Let $P$ be such that $\int_{\R^d} |x|^r dP(x)<+\infty$. For any distribution $\nu$ and any Borel function $g_{\varepsilon} : \R^d\rightarrow \R_+$, $\varepsilon \in (0,\frac13)$, satisfying~$(\ref{criterenuintro})$,
	\begin{equation}
	\label{gepsilongeneral}
	\forall n \geq 2, \quad e_r(a^{(n)},P) \leq \varphi_r(\varepsilon)^{-\frac1d} V_d^{-\frac1d} \left(\frac rd \right)^{\frac1d}\left(\int g_{\varepsilon}^{-\frac rd} dP  \right)^{\frac1r} (n-1)^{-\frac1d}
	\end{equation}
	where $\ds \varphi_r(u)=\left(\frac{1}{3^r}-u^r\right)u^d$.
\end{thm}
\noindent Considering appropriate auxiliary distributions $\nu$ and ``companion'' functions $g_{\varepsilon}$ satisfying $(\ref{criterenuintro})$ yields a Pierce type and a hybrid Zador-Pierce type $L^r$-rate optimality results as established in $\cite{papiergreedy}$  (Zador type results are established in \cite{LuPa15}).\\

Now, we give a micro-macro inequality established in $\cite{mismatch08}$ (see proof of Theorem $2$) to estimate the increments $e_r(\Gamma^n,P)^r - e_r(\Gamma^{n+1},P)^r $, where $(\Gamma^n)_{n\geq 1}$ is a sequence of $L^r$-optimal quantizers of $P$. For every $n \geq 1$,
	\begin{equation}
	\label{increment}
	e_r(\Gamma^n,P)^r - e_r(\Gamma^{n+1},P)^r  \leq \frac{4(2^r-1)e_r(\Gamma^{n+1},P)^r}{n+1}+\frac{4.2^rC_2^r n^{-\frac rd}}{n+1}
	\end{equation}
	where $C_2$ is a finite constant independent of $n$. \\
		
The following Proposition provides a micro-macro inequality established in \cite{papiergreedy} for any quantizer $\Gamma$ of $X$ with distribution $P$.
\begin{prop}
	\label{incopt}
	Assume $\int |x|^rdP(x)<+\infty$. Let $y \in \R^d$ and $\Gamma \subset \R^d$ be a finite quantizer of a random variable $X$ with distribution $P$ such that ${\rm card}(\Gamma)\geq 1$. Then, for every probability distribution $\nu$ on $(\R^d,\mathcal{B}(\R^d))$, every $c \in (0, \tfrac 12)$ 
	\begin{align*}
	e_r(\Gamma,P)^r - e_r(\Gamma \cup \{y\},P)^r 
	\geq \frac{(1-c)^r-c^r}{(c+1)^r} \int \nu \left(B\Big(x, \frac{c}{c+1}d\left(x,\Gamma\right)  \Big) \right) d\left(x,\Gamma\right)^r dP(x). 
	\end{align*}
\end{prop}
\medskip
\noindent From this Proposition, one concludes the following either  for $L^r$-optimal quantizers or for greedy sequences:  \smallskip \\
$\rhd$ Since any sequence of $L^r$-optimal quantizers $(\Gamma^n)_{n \geq 1}$ clearly satisfies $e_r(\Gamma^{n+1},P)\leq e_r(\Gamma^n \cup \{y\}, P)$ for every $y\in \R^d$, then
\begin{align}
\label{micromacrooptimal}
	e_r(\Gamma^n,P)^r - e_r(\Gamma^{n+1},P)^r \geq &e_r(\Gamma^n,P)^r - e_r(\Gamma^n \cup \{y\},P)^r \nonumber \\
	\geq & \frac{(1-c)^r-c^r}{(c+1)^r} \int \nu \left(B\left(x, \frac{c}{c+1}d\left(x,\Gamma^n\right)  \right) \right) d\left(x,\Gamma^n\right)^r dP(x). 
	\end{align}
	$\rhd$ Likewise, since the greedy quantization sequence $(a_n)_{n \geq 1}$ satisfies $e_r(a^{(n+1)},P) \leq e_r(a^{(n)}\cup \{y\},P)$ for every $y \in \R^d$, then 
	\begin{align}
 	\label{micromacrogreedy}
 	e_r(a^{(n)},P)^r - e_r(a^{(n+1)},P)^r 
 	\geq \frac{(1-c)^r-c^r}{(c+1)^r} \int \nu \left(B\left(x, \frac{c}{c+1}d\left(x,a^{(n)}\right)  \right) \right) d\left(x,a^{(n)}\right)^r dP(x). 
 	\end{align}
 	\section{Upper estimates for greedy quantizers}
 	\label{greedydilat}
 	This is the main part of this paper. Let $r,s>0$ and let $(a_n)_{n \geq 1}$ be an $L^r(\R^d)$-optimal greedy quantization sequence of a random variable $X$ with probability distribution $P$. We denote $a^{(n)}=\{a_1,\ldots,a_n\}$ the first $n$ terms of this sequence. For every $\mu \in \R^d$ and $\theta>0$, we denote $a^{(n)}_{\theta,\mu}=\mu+\theta(a^{(n)}-\mu)=\{\mu+\theta(a_i-\mu), \; 1\leq i \leq n\}$. In this section, we study the $L^s$-optimality of the sequence $a^{(n)}_{\theta,\mu}$. \\

 For this, we consider auxiliary probability distributions $\nu$ satisfying the following control on balls with respect to an $L^r$-median $a_1$ of $P$: for every $\varepsilon \in (0,1)$, there exists a Borel function $g_{\varepsilon}: \R^d \rightarrow (0,+\infty)$ such that, for every $x \in \mbox{supp}(\P)$ and every $t \in [0,\varepsilon |x-a_1|]$, 
\begin{equation}
\label{criterenu}
\nu(B(x,t)) \geq g_{\varepsilon}(x) V_d t^d.
\end{equation}
Note that $a_1 \in a^{(n)}$ for every $n \geq 1$ by construction of the greedy quantization sequence so that $d(x,a^{(n)}) \leq d(x,a_1)$ for every $x \in \R^d$.
 	\subsection{Main results}
 	The following result is an avatar of Pierce's Lemma for the $L^s$-error $e_s(a^{(n)}_{\theta,\mu},P) $.
 	\begin{thm}
	\label{Pierce1}
Let $s\in [r,d+r)$ and $1-q=\frac{d+r}{d+r-s}$. Let $(a_n)_{n \geq 1}$ be an $L^r(\R^d)$-optimal greedy quantization sequence of an $\R^d$-valued random variable $X$ with distribution $P=f.\lambda_d$ such that $\E|X|^{r+\delta}<+\infty$ for some $\delta>0$ such that $r+\delta>\frac{sd}{d+r-s}$. Let $\eta \in \big(0,r+\delta-\frac{sd}{d+r-s}\big)$ and let $p'=\frac{r+\delta-\eta}{d|q|}$, $q'=\frac{r+\delta-\eta}{r+\delta-\eta-d|q|}>1$ be two conjugate coefficients larger than $1$. Assume 
 \begin{equation}
 \label{condslarge}
 \int_{\{f>0\}}\left( \frac{f_{\theta,\mu}}{f}\right)^{(1-q)q'}fd\lambda_d <+\infty.
 \end{equation} Then, for every $n \geq 3$,
 \begin{equation}
 e_s(a^{(n)}_{\theta,\mu},P) \leq \theta^{1+\frac{d}{s}} \kappa_{\theta,\mu}^{\text{Greedy,Pierce}} \left(\int_{\{f>0\}}\left( \frac{f_{\theta,\mu}}{f}\right)^{(1-q)q'}fd\lambda_d \right)^{\frac{1}{q'|q|(d+r)}} \sigma_{r+\delta}(P)(n-2)^{-\frac1d}. 
 \end{equation}
where  $e_{r+\delta}(a^{(1)},P)=\sigma_{r+\delta}(P) < +\infty$ denotes the $L^{r+\delta}$-standard deviation of $P$ and
  \begin{align*}
 \kappa_{\theta,\mu}^{\text{Greedy,Pierce}}=& \;2^{\frac1d+\frac{r+\delta}{r+d}(1+\frac{1}{|q|p'}) }  V_d^{-\frac1d} \Big(\frac rd \Big)^{\frac{r}{d(d+r)}}\min_{\varepsilon\in (0,\frac13)}\Big[(1+\ve)\varphi_r(\varepsilon)^{-\frac{1}{d}}\Big] \left(\int (1 \vee |x|)^{\frac{r+\delta}{r+\delta-\eta}} dx \right)^{\frac1d}. 
 \end{align*}
\end{thm}
 When $s \in (0,r]$, notice that
$$e_s(a_{\theta,\mu}^{(n)},P) \leq e_r(a_{\theta,\mu}^{(n)},P)$$
where $e_r(a_{\theta,\mu}^{(n)},P)$ is upper bounded as in Theorem $\ref{Pierce1}$. However, we are still interested in establishing a specific study for $s\in (0,r)$ and giving an upper bound for the $L^s$-error in the following theorem.
\begin{thm}
	\label{Pierce2}
	Let $s<r$ and $X$ be a random variable in $\R^d$ with distribution $P=f.\lambda_d$ such that $\E|X|^{r+\delta}<+\infty$ for some $\delta>0$. Assume 
	$$\int_{\{f>0\}} f^{-\frac{s}{r-s}}f_{\theta,\mu}^{\frac{r}{r-s}}d\lambda_d<+\infty.$$
	Then, for every $n \geq 3$,
	\begin{equation}
	e_s(a^{(n)}_{\theta,\mu},P) \leq  \widetilde{\kappa}_{\theta,\mu}^{\text{Greedy,Pierce}} \theta^{1+\frac ds} 	\left( \int_{\{f>0\}} f^{-\frac{s}{r-s}}f_{\theta,\mu}^{\frac{r}{r-s}}d\lambda_d \right)^{\frac{r-s}{sr}} \sigma_{r+\delta}(P)  (n-2)^{-\frac1d}  
	\end{equation}
	where $e_{r+\delta}(a^{(1)},P)=\sigma_{r+\delta}(P)<+\infty$ and
	\begin{align*}
	\widetilde{\kappa}_{\theta,\mu}^{\text{Greedy,Pierce}}=& \;2^{1+\frac1d+\frac{\delta}{r} }  V_d^{-\frac1d} \Big(\frac rd \Big)^{\frac{r}{d(d+r)}}\min_{\varepsilon\in (0,\frac13)}\Big[(1+\varepsilon)\varphi_r(\varepsilon)^{-\frac{1}{d}}\Big] \left(\int (1 \vee |x|)^{-d(1+\frac {\delta}{r})} dx \right)^{-\frac{1}{d}}. 
	\end{align*}
\end{thm}
 \subsubsection{Application to radial densities}
\label{radialdilat}
In this section, we consider probability distributions with radial densities. 
In other words, if the random variable $X$ has distribution $P=f.\lambda_d$, we consider the auxiliary distribution
$$\nu=\frac{f^{a}}{\int f^{a}d\lambda_d}.\lambda_d:=f_{a}. \lambda_d$$
for $a \in (0,1)$ where the density function $f$ is radial with non-increasing tails w.r.t. $a_1 \in A$ who is peakless w.r.t. $a_1$. These two terms are defined as follows
\begin{defi}
	$(a)$ Let $A \subset \R^d$. A function $f:\R^d \rightarrow \R_+$ is said to be almost radial non-increasing on A w.r.t. $a \in A$ if there exists a norm $\|.\|_0$ on $\R^d$ and real constant $ M\in (0,1]$ such that
	\begin{equation}
	\label{radialtails}
	\forall x \in A\setminus \{a\}, \quad f_{|B_{\|.\|_0}(a,\|x-a\|_0) } \geq Mf(x).  
	\end{equation}
	If $(\ref{radialtails})$ holds for $M=1$, then $f$ is called radial non-increasing on $A$ w.r.t. $a$. \\
	$(b)$ A set $A$ is said to be star-shaped and peakless with respect to $a_1$ if 
	\begin{equation}
	\label{peakless}
	\mathfrak{p}(A,|\cdot-a_1|) := \inf \left\{ \frac{\lambda_d(B(x,t)\cap A)}{\lambda_d(B(x,t))}; x \in A,0<t< |x-a_1| \right\} >0
	\end{equation}
	for any norm  $|\cdot|$ on $\R^d$.
\end{defi}
\begin{rmq}
	\noindent $(a)$~$(\ref{radialtails})$ reads $f(y) \geq M f(x)$ for all $x,y \! \in A\setminus\{a\}$
	for which $\|y-a\|_0 \leq \|x-a\|_0$.
		
		\smallskip
		\noindent $(b)$ If $f$ is radial non-increasing on $\R^d$ w.r.t. $a \in \R^d$ with parameter $\|.\|_0$,
		then there exists a non-increasing measurable function $g:(0, +\infty) \rightarrow \R_+$ satisfying $f(x) = g(\|x -a\|_0 )$ for every $x\neq a$.
		
		\smallskip
		\noindent $(c)$ From a practical point of view, many classes of distributions satisfy~$(\ref{radialtails})$, e.g. the $d$-dimensional normal distribution $\mathcal{N}(m,\sigma_d)$ for which one considers $h(y)=\frac{1}{(2\pi)^{\frac d2}\mbox{\rm det}(\sigma_d)^{\frac12}}e^{-\frac{y^2}{2}}$ and density $f(x)= h(\|x-m\|_0)$ where $\|x\|_0=|\sigma_d^{-\frac12}x|$, and the family of distributions defined by $f(x) \propto |x|^c e^{-a|x|^b}$, for every $x \in \R^d, a,b >0$ and $c>-d$, for which one considers $h(u)=u^ce^{-au^b}$. In the one dimensional case, we can mention the Gamma distribution, the Weibull distributions, the Pareto distributions and the log-normal distributions.\\
		\smallskip
		$(d)$ If $A=\R^d$, then $\mathfrak{p}(A,|\cdot-a|)=1$ for every $a \in \R^d$.

\smallskip
\noindent  $(e)$ The most typical unbounded sets satisfying~$(\ref{peakless})$ are convex cones that is cones  $K \subset \R^d$  of vertex $0$ with $0 \in K$ ($K\neq \varnothing$) and such that $\lambda x \in K$ for every $x \in K$ and $\lambda \geq 0$. For such convex cones $K$  with $\lambda_d(K) >0$, we even have that the lower bound 
		$$
		\mathfrak{p}(K) := \inf \left\{\frac{\lambda_d(B(x,t)\cap K)}{\lambda_d(B(x,t))}; x \!\in K,\,t>0 \right\} = \frac{\lambda_d\big(B(0,1)\cap K) \big)}{V_d}>0.
		$$
		Thus if $K= \R_+^d$, then $\mathfrak{p}(K)= 2^{-d}$.
\end{rmq}
\begin{thm}
		\label{zadorpierce}
		Let $s\in [r,d+r)$ and  $1-q=\frac{d+r}{d+r-s}$. Assume that $P=f. \lambda_d$ has finite polynomial moments of order $\frac{(1-a)(d+\ve)}{a}$ for some $a \in (0,1)$ and $\ve>0$. Let $a_1$ denote the $L^r$-median of $P$ and assume that $\mbox{supp}(P) \subset A$ and $a_1 \in A$ for some $A$ star-shaped and peakless with respect to $a_1$ and that $f$ is almost radial non-increasing with respect to $a_1$ in the sense of $(\ref{radialtails})$.  Assume 
		\begin{equation}
		\label{condradial}\int_{\{f>0\}} f^{\frac{-s(1+a)}{d+r-s}} f_{\theta,\mu}^{\frac{d+r}{d+r-s}} d\lambda_d<+\infty.
		\end{equation}
		Then, for every $n \geq 3$, 
		$$ e_s(a^{(n)}_{\theta,\mu},P) \leq \kappa_{\theta,\mu}^{\text{G,Z,P}} \; \theta^{1+\frac{d}{s}}\|f\|_{\frac{d}{d+r}}^{\frac{1}{d+r}}\|f\|_a^{\frac{a}{d+r}} \left(\int_{\{f>0\}} f^{\frac{-s(1+a)}{d+r-s}} f_{\theta,\mu}^{\frac{d+r}{d+r-s}} d\lambda_d \right)^{\frac{1}{|q|(d+r)}} (n-2)^{-\frac{1}{d}},
		$$	
		where
		$\kappa_{\theta,\mu}^{\text{G,Z,P}} \leq  \frac{2^{1+\frac1d}C_0^2\,r^{\frac 1d}\,}{d^{\frac1d}M^{\frac1d}V_d^{\frac1d}\mathfrak{p}(A,|\cdot-a_1|)^{\frac1d}} \min_{\varepsilon \in (0,\frac13)} \Big[\varphi_r(\varepsilon)^{-\frac 1d}\Big].$
\end{thm}
\begin{rmq}
\label{remsura}
Note that the condition $(\ref{condradial})$ is more restrictive than the condition $(\ref{condslarge})$ in a sense that the set of values of $\theta$ for which $(\ref{condradial})$ is satisfied is smaller than the set for which $(\ref{condslarge})$ is satisfied. This will be made precise and clear in Section $\ref{examplesdilat}$ for particular distributions. \\
However, if $P$ has finite polynomial moments of any order $r>0$, i.e. the parameter $a$ in Theorem $\ref{zadorpierce}$ being as small as possible $(a \rightarrow 0^+)$, then the condition $(\ref{condradial})$ yield the same interval as $(\ref{condslarge})$.
\end{rmq}
 	\subsection{Proofs}
 	\subsubsection{General results}
 	We first state two rather theoretical results based on the auxiliary distribution $\nu$ and its companion function $g_{\ve}$ satisfying $(\ref{criterenu})$. More operating criterions based on moments of $P$ and/or the radial structure of its densities will appear as consequences of these Theorems by specifying the distribution $\nu$ (and $g_{\ve}$).
 	\begin{thm}
\label{thmgreedygeneral1}
	Let $s\in [r,d+r)$ and $1-q=\frac{d+r}{d+r-s}$. Let $(a_n)_{n \geq 1}$ be an $L^r(\R^d)$-optimal greedy quantization sequence of an $\R^d$-valued random variable $X$ with distribution $P=f.\lambda_d$ such that $\E|X|^{r+\delta}<+\infty$ for some $\delta>0$. Assume  there exists an auxiliary distribution $\nu$ and a Borel function $g_{\ve}$ satisfying $(\ref{criterenu})$ for $\ve \in (0,\tfrac 13)$ such that
	$$\int_{\{f>0\}} \left(\frac{f_{\theta,\mu}}{f\,g_{\ve}}\right)^{|q|} dP_{\theta,\mu}(x)<+\infty.$$
	Then, for every $n \geq 3$, 
	\begin{equation}
	\label{bornegreedygeneral}
	e_s(a^{(n)}_{\theta,\mu},P) \leq \theta^{1+\frac{d}{d+r}} \kappa_{\theta,\mu}^{greedy}\left(\int g_{\varepsilon}^{-\frac{r}{d}}dP \right)^{\frac{1}{d+r}}  \left(\int_{\{f>0\}} \left(\frac{f_{\theta,\mu}}{f\,g_{\ve}}\right)^{|q|}  dP_{\theta,\mu}(x) \right)^{\frac{1}{|q|(d+r)}} (n-2)^{-\frac{1}{d}} 
	\end{equation}
	where $\kappa_{\theta,\mu}^{greedy}= 2^{\frac{1}{d}} V_d^{-\frac{1}{d}} \big(\frac rd\big)^{\frac{r}{d(d+r)}} \min_{\varepsilon\in (0,\frac13)}\Big[\varphi_r(\varepsilon)^{-\frac{1}{d}}\Big]$.
\end{thm}
\begin{proof}
We start by noticing that, for every $n \geq 1$,
	\begin{align*}
	e_s(a^{(n)}_{\theta,\mu},P)^s=& \int_{\R^d} d(z, a^{(n)}_{\theta,\mu})^s f(z) d\lambda_d(z)
	=   \int_{\R^d} \min_{x_i\in a^{(n)}, 1 \leq i \leq n}\big|z-\mu+\theta (\mu -x_i)\big|^s f(z) d\lambda_d(z).
	\end{align*}
	Then, by applying the change of variables $x=\frac{z-\mu}{\theta}+\mu $, one obtains
	\begin{align}
	\label{linkthetamugreedy}
	e_s(a^{(n)}_{\theta,\mu},P)^s=& \theta^{s+d} \int_{\R^d} d(x,a^{(n)})^sf(\mu+\theta(x-\mu)) d\lambda_d(x)\nonumber \\
	=&\theta^s \int_{\R^d} d(x,a^{(n)})^s dP_{\theta,\mu}(x)\nonumber \\
	= &   \theta^s e_s(a^{(n)},P_{\theta,\mu})^s.
	\end{align} 
	Now, let us study $e_s(a^{(n)},P_{\theta,\mu})$. Consider $c \in (0 , \frac{\varepsilon}{1-\varepsilon}] \cap (0,\frac12)$ so that $\frac{c}{c+1} \leq \varepsilon$. Hence, for any such $c$, $\ds \frac{c}{c+1} \, d(x,a^{(n)}) \leq \varepsilon |x-a_1|$ since $a_1 \in a^{(n)}$. Consequently, criteria $(\ref{criterenu})$ is satisfied, so there exists a function $g_{\varepsilon}$ such that 
	$$\nu \left(B\left(x, \frac{c}{c+1} \,d\big(x,a^{(n)}\big)  \right) \right) \geq V_d \, \left(\frac{c}{c+1} \right)^{d} d(x,a^{(n)})^d\, g_{\varepsilon}(x).$$
	Then, noticing that $ \frac{(1-c)^r-c^r}{(1+c)^r} \geq \frac{1}{3^r} -\big( \frac{c}{c+1}\big)^r \; >0$, since $c \in (0,\tfrac 12)$, $(\ref{micromacrogreedy})$ yields
	\begin{equation}
	\label{equ1}
	e_r(a^{(n)},P)^r - e_r(a^{(n+1)},P)^r  
	\geq V_d \, \varphi_r\left(\frac{c}{c+1} \right) \int g_{\varepsilon}(x) d(x,a^{(n)})^{d+r} dP(x)
	\end{equation}  
	where $\ds \varphi_r(u)=\left(\frac{1}{3^r}-u^r\right)u^d, \; u \in (0,\tfrac13)$. Consequently, 
	\begin{equation*}
	e_r(a^{(n)},P)^r - e_r(a^{(n+1)},P)^r  
	\geq V_d \, \varphi_r\left(\frac{c}{c+1} \right) \theta^{-d} \int_{\R^d} g_{\varepsilon}(x) d(x,a^{(n)})^{d+r} f(x) f_{\theta,\mu}^{-1}(x)dP_{\theta,\mu}(x).
	\end{equation*}  
	Now, applying the reverse H\"older inequality with conjugate exponents $p= \frac{s}{d+r} \in (0,1)$ and $q=\frac{-s}{d+r-s}<0$ yields
	\begin{align}
	\label{incutile}
	e_r(a^{(n)},P)^r - e_r(a^{(n+1)},P)^r  
	\geq & \, V_d \, \varphi_r\left(\frac{c}{c+1} \right) \theta^{-d} \left(\int_{\{f>0\}} \big(g_{\varepsilon}(x) f(x) f_{\theta,\mu}^{-1}(x)\big)^q dP_{\theta,\mu}(x) \right)^{\frac1q}\nonumber \\
	& \times \left(\int_{\R^d}d(x,a^{(n)})^{s} dP_{\theta,\mu}(x) \right)^{\frac1p}\nonumber\\
	\geq &\, V_d \, \varphi_r\left(\frac{c}{c+1} \right) \theta^{-d} \left(\int_{\{f>0\}} \left(\frac{f_{\theta,\mu}}{f\,g_{\ve}}\right)^{|q|} (x) dP_{\theta,\mu}(x) \right)^{\frac1q} e_s(a^{(n)},P_{\theta,\mu})^{d+r}.
	\end{align} 
	Consequently, denoting $\ds C_1=V_d \, \varphi_r\left(\frac{c}{c+1} \right) \theta^{-d}  \left(\int_{\{f>0\}} \left(\frac{f_{\theta,\mu}}{f\,g_{\ve}}\right)^{|q|} (x) dP_{\theta,\mu}(x) \right)^{\frac1q}$, one obtains
	\begin{equation}
	\label{incutile}
	e_r(a^{(n)},P)^r - e_r(a^{(n+1)},P)^r  \geq C_1 e_s(a^{(n)},P_{\theta,\mu})^{d+r}.
	\end{equation}
	At this stage, we know that $e_r(a^{(k)},P)$ is decreasing w.r.t $k$ and it is clear that it is the same for $e_s(a^{(k)},P_{\theta,\mu})$, since 
	$$e_s(a^{(k)},P_{\theta,\mu})= \E\left[\min_{1 \leq i \leq k} |a_i-\frac{X-\mu}{\theta}-\mu|^s\right]^{\frac1s}\geq \E\left[\min_{1 \leq i \leq k+1} |a_i-\frac{X-\mu}{\theta}-\mu|^s\right]^{\frac1s}=e_s(a^{(k+1)},P_{\theta,\mu}),$$
	so, one has 
	\begin{align*}
	n\, e_s(a^{(2n-1)},P_{\theta,\mu})^{d+r} \leq \sum_{k=n}^{2n-1} e_s(a^{(k)},P_{\theta,\mu})^{d+r} \leq \frac{1}{C_1} \sum_{k=n}^{2n-1} e_r(a^{(k)},P)^r - e_r(a^{(k+1)} ,P)^r \leq \frac{1}{C_1} e_r(a^{(n)},P)^r.	
	\end{align*}
	and, since $\ds 2 \left\lceil \frac{n}{2} \right\rceil -1 \leq n$, 
	$$\frac{n}{2} e_s(a^{(n)},P_{\theta,\mu})^{d+r} \leq \left\lceil \frac{n}{2} \right\rceil e_s(a^{(n)},P_{\theta,\mu})^{d+r} \leq  \left\lceil \frac{n}{2} \right\rceil e_s\left(a^{2 \left\lceil \frac{n}{2} \right\rceil -1},P_{\theta,\mu}\right)^{d+r} \leq \frac{1}{C_1}e_r\left(a^{ \left\lceil \frac{n}{2} \right\rceil },P\right)^{r}.$$
	Consequently, using the result of Theorem $\ref{greedyrateoptimal}$ 
	\begin{align*}
	e_s(a^{(n)},P_{\theta,\mu})  \leq & \left(\frac{2}{C_1}\right)^{\frac{1}{d+r}} n^{-\frac{1}{d+r}} e_r\left(a^{ \left\lceil \frac{n}{2} \right\rceil },P\right)^{\frac{r}{d+r}} \\
	\leq & \, 2^{\frac{1}{d}} V_d^{-\frac{1}{d}} \big(\frac rd\big)^{\frac{r}{d(d+r)}} \varphi_r\left(\frac{c}{c+1}\right)^{-\frac{1}{d}} \theta^{\frac{d}{d+r}} \left(\int_{\R^d} g_{\varepsilon}^{-\frac{r}{d}}dP \right)^{\frac{1}{d+r}} \\
	& \times \left(\int_{\{f>0\}} \left(\frac{f_{\theta,\mu}}{f\,g_{\ve}}\right)^{|q|} (x) dP_{\theta,\mu}(x) \right)^{\frac{1}{|q|(d+r)}} (n-2)^{-\frac{1}{d}}.
	\end{align*}
	We are led to study $\varphi_r\left( \frac{c}{c+1}\right)^{-\frac 1d}$ subject to the constraint $c \in \big(0,\frac{\varepsilon}{1-\varepsilon}\big] \cap \big(0,\frac12\big)$.
	$\varphi_r$ is increasing in the neighborhood of $0$ and $\varphi_r(0)$, so, one has, for every $\varepsilon \in (0,\frac 13)$ small enough,
	$\varphi_r\left( \frac{c}{c+1}\right) \leq \varphi_r(\varepsilon), \mbox{ for } c \in (0,\tfrac{\varepsilon}{1-\varepsilon}].$ 
	This leads to specify $c$ as  
	$ c=\frac{\varepsilon}{1-\varepsilon} \mbox{, so that}  \frac{c}{c+1}=\varepsilon$ which means that one can use 
	\begin{align}
	\label{varphimin}
	\varphi_r\left( \frac{c}{c+1}\right)^{-\frac{1}{d+r}} \leq  \min_{\varepsilon \in (0,\tfrac13)}\Big[\varphi_r(\varepsilon)^{-\frac{1}{d+r}} \Big]
	\end{align}
	 which yields
	 \begin{align}
	 \label{etoile}
	e_s(a^{(n)},P_{\theta,\mu})  \leq 	&\, 2^{\frac{1}{d}} V_d^{-\frac{1}{d}} \big(\frac rd\big)^{\frac{r}{d(d+r)}} \min_{\varepsilon \in (0,\frac13)}\big[\varphi_r(\varepsilon)^{-\frac{1}{d}}\big] \theta^{\frac{d}{d+r}} \left(\int_{\R^d} g_{\varepsilon}^{-\frac{r}{d}}dP \right)^{\frac{1}{d+r}}  \nonumber \\
	& \times \left(\int_{\{f>0\}} \left(\frac{f_{\theta,\mu}}{f\,g_{\ve}}\right)^{|q|}(x) dP_{\theta,\mu}(x) \right)^{\frac{1}{|q|(d+r)}} (n-2)^{-\frac{1}{d}}.
	 \end{align}
	 Finally, one concludes by merging this with $(\ref{linkthetamugreedy})$.
	\hfill $\square$\\
\end{proof}
\begin{thm}
	\label{thmgreedygeneral2}
	Let $s<r$ and $X$ a random variable in $\R^d$ with distribution $P=f.\lambda_d$ and such that $\E|X|^{r+\delta}<+\infty$ for some $\delta>0$. Assume there exists an auxiliary distribution $\nu$ and a Borel function $g_{\ve}$ satisfying $(\ref{criterenu})$ for every $\ve\in (0,\tfrac 13)$ such that 
	$$\int_{\R^d} g_{\varepsilon}^{-\frac{r}{d}}dP <+\infty\quad \mbox{and} \quad  \int_{\{f>0\}} f^{-\frac{s}{r-s}}f_{\theta,\mu}^{\frac{r}{r-s}}d\lambda_d <+\infty.$$ 
	Then, for every $n \geq 3$, 
	\begin{equation}
	\label{bornegreedygeneral2}
	e_s(a^{(n)}_{\theta,\mu},P) \leq \theta^{1+\frac ds} \kappa_{\theta,\mu}^{\text{Greedy}}\left(\int_{\R^d} g_{\varepsilon}^{-\frac{r}{d}}dP \right)^{\frac{1}{r}}  \left( \int_{\{f>0\}} f^{-\frac{s}{r-s}}f_{\theta,\mu}^{\frac{r}{r-s}}d\lambda_d \right)^{\frac{r-s}{sr}}  (n-2)^{-\frac{1}{d}}
	\end{equation}
	where $\kappa_{\theta,\mu}^{\text{Greedy}}= 2^{1+\frac{1}{d}} V_d^{-\frac{1}{d}} \big(\frac rd\big)^{\frac{r}{d(d+r)}} \min_{\varepsilon\in (0,\frac13)}\big[\varphi_r(\varepsilon)^{-\frac{1}{d}}\big]$.
\end{thm}
\begin{proof}
	We start from Equation $(\ref{incutile})$ in the proof of Theorem $\ref{thmgreedygeneral1}$ recalled below 
	$$e_r(a^{(n)},P)^r - e_r(a^{(n+1)},P)^r  
	\geq C_1 e_s(a^{(n)},P_{\theta,\mu})^{d+r}$$
	where $\ds C_1=\varphi_r\left(\frac{c}{c+1} \right) \theta^{-d+\frac dq} \left(\int_{\{f>0\}} g_{\varepsilon}^q(x) f^q(x) f_{\theta,\mu}^{1-q}(x) d\lambda_d(x) \right)^{\frac1q}$ and $q=-\frac{s}{d+r-s}<0$ so that $1-q=\frac{d+r}{d+r-s}$.
	At this stage, follow the lines of the proof of Theorem $\ref{thmgreedygeneral1}$ to get, for $n \geq 3$,
	\begin{align*}
	e_s(a^{(n)},P_{\theta,\mu})  \leq & \left(\frac{2}{C_1}\right)^{\frac{1}{d+r}} (n-1)^{-\frac{1}{d+r}} e_r\left(a^{ \left\lceil \frac{n}{2} \right\rceil },P\right)^{\frac{r}{d+r}} \\
	\leq & \, \kappa_{\theta,\mu}^{\text{Greedy}} \theta^{\frac ds} \left(\int_{\R^d} g_{\varepsilon}^{-\frac{r}{d}}dP \right)^{\frac{1}{d+r}} \left( \int_{\{f>0\}} g_{\varepsilon}^q f^q f_{\theta,\mu}^{1-q }d\lambda_d \right)^{\frac{1}{|q|(d+r)}} (n-2)^{-\frac{1}{d}}.
	\end{align*}
	where $\kappa_{\theta,\mu}^{\text{Greedy}}= 2^{1+\frac{1}{d}} V_d^{-\frac{1}{d}} \big(\frac rd\big)^{\frac{r}{d(d+r)}} \min_{\varepsilon\in (0,\frac13)}\big[\varphi_r(\varepsilon)^{-\frac{1}{d}}\big]$.\\
	Now, since $s<r$, one can apply H\"older inequality with the conjugate exponents $p'=\frac{r(d+r-s)}{r(d+r-s)-ds}>1$ and $q'=\frac{r}{d|q|}=\frac{r(d+r-s)}{ds}>1$ which yields
	\begin{align*}
	\int_{\{f>0\}} g_{\varepsilon}^q f^q f_{\theta,\mu}^{1-q }d\lambda_d=\int_{\{f>0\}} g_{\varepsilon}^q f^{q-1} f_{\theta,\mu}^{1-q }dP
	\leq  \left( \int_{\R^d} g_{\varepsilon}^{-\frac rd}dP \right)^{\frac{1}{q'}} \left( \int_{\{f>0\}} f^{\frac{r}{s-r}+1} f_{\theta,\mu}^{\frac{r}{r-s}} d\lambda_d \right)^{\frac{1}{p'}}
	\end{align*}
	so 
	$$\left( \int_{\{f>0\}} g_{\varepsilon}^q f^q f_{\theta,\mu}^{1-q }d\lambda_d \right)^{-\frac{1}{q(d+r)}} \leq \left( \int_{\R^d} g_{\varepsilon}^{-\frac rd}dP \right)^{\frac{d}{r(d+r)}}\left( \int_{\{f>0\}} f^{\frac{s}{s-r}} f_{\theta,\mu}^{\frac{r}{r-s}} d\lambda_d \right)^{\frac{r-s}{rs}}$$
	and 
	$$e_s(a^{(n)},P_{\theta,\mu})  \leq \, \kappa_{\theta,\mu}^{\text{Greedy}} \theta^{\frac ds} \left(\int_{\R^d} g_{\varepsilon}^{-\frac{r}{d}}dP \right)^{\frac{1}{r}} \left( \int_{\{f>0\}} f^{-\frac{s}{r-s}}f_{\theta,\mu}^{\frac{r}{r-s}}d\lambda_d \right)^{\frac{r-s}{sr}} (n-2)^{-\frac{1}{d}}.$$
	and one deduces the result just as in the proof of Theorem $\ref{thmgreedygeneral1}$.
	\hfill $\square$
\end{proof}
 	\subsubsection{Proofs of main results}
 	\begin{refproofgreedylarge}
 	We consider $\nu(dx)=\gamma_{r,\delta}(x) \lambda_d(dx)$ where
	$$\gamma_{r,\delta}(x)=\frac{K_{\delta,r}}{(1 \vee |x-a_1|)^{d\frac{r+\delta}{r+\delta-\eta}}}\;  \qquad \mbox{with} \qquad  K_{\delta,r}= \left(\int \frac{dx}{(1 \vee |x|)^{\frac{r+\delta}{r+\delta-\eta}}} \right)^{-1} <+\infty$$
	is a probability density with respect to the Lebesgue measure on $\R^d$. For every $x \in \R^d$ such that $\varepsilon |x-a_1| \geq t $ and every $y \in B(x,t)$, one has
	$|y-a_1| \leq |y-x|+|x-a_1| \leq (1+\varepsilon) |x-a_1|$ so that 
	$$\nu(B(x,t)) \geq \frac{K_{\delta,r} V_d \,t^d}{\big(1 \vee(1+\varepsilon)|x-a_1|\big)^{d\frac{r+\delta}{r+\delta-\eta}}}.$$ 
	Hence, $(\ref{criterenu})$ is satisfied with $$g_{\varepsilon}(x)=\frac{K_{\delta,r}}{\big(1 \vee (1+\varepsilon)|x-a_1|\big)^{\frac{r+\delta}{r+\delta-\eta}}}$$
	so we apply Theorem $\ref{thmgreedygeneral1}$ where one has to handle the term 
	$$\left(\int_{\{f>0\}} \left(\frac{f_{\theta,\mu}}{f\,g_{\ve}}\right)^{|q|}(x)dP_{\theta,\mu}(x) \right)^{\frac{1}{|q|(d+r)}}=\theta^{\frac{d}{|q|(d+r)}}\left(\int_{\{f>0\}} g_{\varepsilon}^q \left(\frac{f_{\theta,\mu}}{f}\right)^{1-q} dP(x) \right)^{\frac{1}{|q|(d+r)}}$$
where $q=\frac{-s}{d+r-s}<0$ so that $1-q=\frac{d+r}{d+r-s}$. To do this, we apply H\"older inequality with the conjugate coefficients $p'=\frac{r+\delta-\eta}{d|q|}>1$ (due to the moment assumption on $P$) and $q'=\frac{r+\delta-\eta}{r+\delta-\eta-d|q|}>1$. This yields
	\begin{align*}
	\left(\int_{\{f>0\}} g_{\varepsilon}^q \left(\frac{f_{\theta,\mu}}{f}\right)^{1-q} dP \right)^{\frac{1}{|q|(d+r)}}\leq & \left(\int_{\R^d}g_{\ve}^{qp'}dP\right)^{\frac{1}{p'|q|(d+r)}} \left( \int_{\{f>0\}}\left(\frac{f_{\theta,\mu}}{f}\right)^{(1-q)q'}dP\right)^{\frac{1}{q'|q|(d+r)}}\\
	\leq & \left(\int_{\R^d}g_{\ve}^{qp'}dP\right)^{\frac{1}{p'|q|(d+r)}} \left( \int_{\{f>0\}}\left(\frac{f_{\theta,\mu}}{f}\right)^{(1-q)q'}fd\lambda_d \right)^{\frac{1}{q'|q|(d+r)}}
	\end{align*}
	so that
	\begin{align}
	\label{holderbigintegral}
	\left(\int_{\{f>0\}} \big(g_{\varepsilon}(x) f(x) f_{\theta,\mu}^{-1}(x)\big)^q dP_{\theta,\mu}(x) \right)^{\frac{1}{|q|(d+r)}} &\leq \theta^{\frac{d}{|q|(d+r)}} \left(\int_{\R^d}g_{\ve}^{qp'}dP\right)^{\frac{1}{p'|q|(d+r)}} \nonumber \\
	& \quad \times \left( \int_{\{f>0\}}\left(\frac{f_{\theta,\mu}}{f}\right)^{(1-q)q'}fd\lambda_d \right)^{\frac{1}{q'|q|(d+r)}}.
	\end{align}
	Consequently,	
	\begin{align*}
	e_s(a^{(n)}_{\theta,\mu},P)& \leq  \theta^{1+\frac{d}{s}} \kappa_{\theta,\mu}^{\text{Greedy}}\left( \int_{\{f>0\}}\left(\frac{f_{\theta,\mu}}{f}\right)^{(1-q)q'}fd\lambda_d \right)^{\frac{1}{q'|q|(d+r)}}\left(\int_{\R^d} g_{\varepsilon}^{-\frac{r}{d}}dP \right)^{\frac{1}{d+r}} \\
	& \times \left(\int_{\R^d}g_{\ve}^{qp'}dP\right)^{\frac{1}{p'|q|(d+r)}}  (n-2)^{-\frac{1}{d}}
	\end{align*}
	By our choice of $g_{\ve}$, 
	$$\left(\int_{\R^d} g_{\varepsilon}^{-\frac{r}{d}}dP \right)^{\frac{1}{d+r}}\leq  \left(\int_{\R^d}\big(1 \vee (1+\ve)\|x-a_1\| \big)^{\frac{r(r+\delta)}{d(r+\delta-\eta)}} dP \right)^{\frac{1}{d+r}}$$ 
	and $$\left(\int_{\R^d}g_{\ve}^{qp'}dP\right)^{\frac{1}{p'|q|(d+r)}} \leq \left(\int_{\R^d}\big(1 \vee (1+\ve)\|x-a_1\| \big)^{r+\delta} dP \right)^{\frac{1}{p'|q|(d+r)}}.$$
	At this stage, notice that $\ds \frac{r(r+\delta)}{d(r+\delta-\eta)}<r+\delta$ since $r+\delta-\eta>\frac{sd}{d+r-s}>\frac{r}{d}$. So, 
	$$\int_{\R^d}\big(1 \vee (1+\ve)\|x-a_1\| \big)^{\frac{r(r+\delta)}{d(r+\delta-\eta)}} dP <\int_{\R^d}\big(1 \vee (1+\ve)\|x-a_1\| \big)^{r+\delta} dP$$
	since the function $x \mapsto a^x$ is increasing w.r.t $x$ for $a>1$. Moreover, owing to $L^{r+\delta}$-Minkowski inequality, 
	$$\left(\int_{\R^d}\big(1 \vee (1+\ve)\|x-a_1\| \big)^{r+\delta} dP \right)^{\frac{1}{d+r}\big(1+\frac{1}{|q|p'}\big)} \leq \Big(1+(1+\ve)\sigma_{r+\delta}(P)\Big)^{\frac{r+\delta}{r+d}\big(1+\frac{1}{|q|p'}\big)}$$
	where 	$\sigma_{r+\delta}(P)=\inf_{a}\|X-a\|_{r+\delta}$ is the $L^{r+\delta}$-standard deviation of $P$. 	Consequently, 
	\begin{align*}
	 e_s(a^{(n)}_{\theta,\mu},P) \leq &\theta^{1+\frac{d}{s}} \frac{\kappa_{\theta,\mu}^{\text{Greedy}}}{K_{\delta,r}^{\frac1d}}\left(\int_{\{f>0\}}\left(\frac{f_{\theta,\mu}}{f}\right)^{(1-q)q'}fd\lambda_d \right)^{\frac{1}{q'|q|(d+r)}}\Big(1+(1+\ve)\sigma_{r+\delta}(P)\Big)^{\frac{(r+\delta)(1+|q|p')}{|q|p'(r+d)}}(n-2)^{-\frac1d}. 
	\end{align*}
	Now, we introduce an equivariance argument. For $\lambda>0$, let $X_{\lambda}:=\lambda(X-a_1)+a_1$ and $(\alpha_{\lambda,n})_{n\geq 1}:=(\lambda(\alpha_n-a_1)+a_1)_{n \geq 1}$. It is clear that  $e_r(\alpha^{(n)},X)=\frac{1}{\lambda}e_r(\alpha_{\lambda}^{(n)},X_{\lambda})$. 
		Plugging this in the previous inequality yields 
		\begin{align*}
	 e_s(a^{(n)}_{\theta,\mu},P)  & \leq \theta^{1+\frac{d}{s}}\kappa_{\theta,\mu}^{\text{Greedy}}K_{\delta,r}^{-\frac1d}\left(\int_{\{f>0\}}\left(\frac{f_{\theta,\mu}}{f}\right)^{(1-q)q'}fd\lambda_d \right)^{\frac{1}{q'|q|(d+r)}}\\
	 &\quad  \times \frac{1}{\lambda}\Big(1+(1+\ve)\lambda \sigma_{r+\delta}(P)\Big)^{\frac{(r+\delta)(1+|q|p')}{|q|p'(r+d)}}(n-2)^{-\frac1d}. 
	\end{align*}
	Finally, one deduces the result by setting $\ds \lambda= \frac{1}{(1+\varepsilon)\sigma_{r+\delta}}$. 
 \hfill $\square$
\medskip 
\end{refproofgreedylarge}

 \begin{refproofgreedysmall}
	We consider the function $g_{\varepsilon}$ defined by 
	$$g_{\varepsilon}(x)=\frac{K_{\delta,r}}{\big(1 \vee(1+\varepsilon)|x-a_1|\big)^{d(1+\frac {\delta}{r})}}$$
	where $\ds K_{\delta,r}= \left(\int \frac{dx}{(1 \vee |x|)^{d(1+\frac {\delta}{r})}} \right)^{-1} <+\infty$. 
	One has 
	\begin{align*}
	\left(\int_{\R^d}g_{\varepsilon}^{-\frac rd}(x) dP\right)^{\frac{1}{r}}\leq K_{\delta,r}^{-\frac{1}{d}}\left( \int \big( 1 \vee (1+\varepsilon)|x-a_1|\big)^{r+\delta}dP\right)^{\frac{1}{r}}
	\end{align*}
	so that, applying the $L^{r+\delta}$-Minkowski inequality, one obtains	
	$$\left(\int g_{\varepsilon}(x)^{-\frac rd} dP(x) \right)^{\frac{1}{r}} \leq K_{\delta,r}^{-\frac{1}{d}} \left(1 + (1+\varepsilon)\sigma_{r+\delta} \right)^{1+\tfrac{\delta}{r}}.$$
	Then, applying Theorem $\ref{thmgreedygeneral2}$ yields , for every $n \geq 3$,
	\begin{align}
	e_s(a^{(n)}_{\theta,\mu},P) \leq & \; \theta^{1+\frac ds}  \kappa_{\theta,\mu}^{\text{Greedy}} K_{\delta,r}^{-\frac1d}\left(1 + (1+\varepsilon)\sigma_{r+\delta} \right)^{1+\tfrac{\delta}{r}}\left( \int_{\{f>0\}} f^{-\frac{s}{r-s}}f_{\theta,\mu}^{\frac{r}{r-s}}d\lambda_d \right)^{\frac{r-s}{sr}} (n-2)^{-\frac{1}{d}}
	\end{align}
	Finally, using the equivariance argument introduced in the proof of Theorem $\ref{Pierce1}$, one deduces, in the same spirit, the result by considering $\lambda=\frac{1}{(1+\varepsilon)\sigma_{r+\delta}(P)}$.	
	\hfill $\square$
	\bigskip
\end{refproofgreedysmall}

For the proof of Theorem \ref{zadorpierce}, we use the following technical lemma (established in $\cite{papiergreedy}$).
\begin{lem}
	\label{lem2}
	Let $\nu=f . \lambda_d$ be a probability measure on $\R^d$ where $f$ is almost radial non-increasing on $A\in \mathcal{B}(\R^d)$ w.r.t. $a_1 \in A$, $A$ being star-shaped relative to $a_1$ and satisfying $(\ref{peakless})$. Then, for every $x \in A$ and $ t \in (0, |x-a_1|)$,
	$$\nu(B(x,t))\geq M \mathfrak{p}(A,|\cdot-a_1|)(2C_0^2)^{-d} V_d f(x) t^d$$
	where $C_0 \in [1,+\infty)$ is such that, for every $x \in \R^d$, $\ds \frac{1}{C_0}\|x\|_0 \leq |x| \leq C_0 \|x\|_0$.
	\end{lem}
\begin{refproofhybrid}
We consider $\nu=f_a d\lambda_d$ for $a \in (0,1)$  where $$f_a=K_a\,f^a \quad \mbox{with} \quad K_a=\left(\int f^{a}d\lambda_d\right)^{-1}.$$ 
Note that $\int f^ad\lambda_d<+\infty$. In fact, if we denote 
$f^a=f^a(1+|x|)^b(1+|x|)^{-b}$ where $b=(1-a)(d+\ve)$, $\ve>0$, then, applying H\"older's inequality with the conjugate coefficients $\frac1a$ and $\frac{1}{1-a}$ yields
$$\int f^a(x) d\lambda_d(x)\leq \left( \int f(x) \, (1+|x|)^{\frac{1-a}{a}(d+\ve)}d\lambda_d(x)\right)^{a} \left( \int (1+|x|)^{-(d+\ve)}d\lambda_d(x)\right)^{1-a} $$
where the first factor is finite due to the moment assumption made on $P$ and the second factor is finite for $\ve>0$.\\
 
	Let $c \in (0,\tfrac12)$. Since $\frac{c}{c+1}<1$ and $a_1 \in a^{(n)}$ then, for every $x \in \R^d$, $\frac{c}{c+1}d(x,a^{(n)}) \leq d(x,a^{(n)}) \leq |x-a_1|$. Moreover, notice that $f_a$ is radial non-increasing with parameter $M^a$. So, merging $(\ref{micromacrogreedy})$ with Lemma $\ref{lem2}$, one obtains
	$$e_r(a^{(n)},P)^r - e_r(a^{(n+1)},P)^r 
	\geq \varphi_r\left(\frac{c}{c+1}\right)M^a \mathfrak{p}(A,|\cdot-a_1|)(2C_0^2)^{-d} V_d \int f_a(x) d(x,a^{(n)})^{d+r} dP(x).$$
	Now, denoting $C=\varphi_r\left(\frac{c}{c+1}\right)M^a \mathfrak{p}(A,|\cdot-a_1|)(2C_0^2)^{-d} V_d$ and having in mind that $dP=f.d\lambda_d$ and $dP_{\theta,\mu}=\theta^{d}f_{\theta,\mu}.d\lambda_d$, yields 
	\begin{align*}
	e_r(a^{(n)},P)^r - e_r(a^{(n+1)},P)^r &
	\geq C \theta ^{-d} \int_{\{f>0\}} f_a(x) f(x) f_{\theta,\mu}^{-1}(x)  d(x,a^{(n)})^{d+r} dP_{\theta,\mu}(x)\\
	& \geq C \theta ^{-d} K_{a}\int_{\{f>0\}} f(x)^{1+a} f_{\theta,\mu}^{-1}(x)  d(x,a^{(n)})^{d+r} dP_{\theta,\mu}(x).
	\end{align*}
	Applying the reverse H\"older inequality with the conjugate exponents $p= \frac{s}{d+r} \in (0,1)$ and $q=\frac{-s}{d+r-s}<0$ yields
	\begin{align*}
	e_r(a^{(n)},P)^r - e_r(a^{(n+1)},P)^r 
	& \geq  C \theta ^{-d} K_{a} \left(\int_{\{f>0\}} f(x)^{-|q|(1+a)} f_{\theta,\mu}^{|q|}(x) dP_{\theta,\mu}(x) \right)^{\frac1q} \left(\int_{\R^d}d(x,a^{(n)})^{s} dP_{\theta,\mu}(x) \right)^{\frac{d+r}{s}}\\
	& \geq  C \theta ^{-d+\frac dq} K_{a} \left(\int_{\{f>0\}} f(x)^{-|q|(1+a)} f_{\theta,\mu}^{1-q}(x) d\lambda_d(x) \right)^{\frac1q} e_s(a^{(n)},P_{\theta,\mu})^{d+r}.
	\end{align*}
	At this stage, we denote $C_1=C \theta ^{-d+\frac dq} K_{a} \left(\int_{\{f>0\}} f(x)^{-|q|(1+a)} f_{\theta,\mu}^{1-q}(x) d\lambda_d(x) \right)^{\frac1q}$, follow the same steps as in the proof of Theorem $\ref{thmgreedygeneral1}$ and use the result of Theorem $2.2.8$ in $\cite{papiergreedy}$ to obtain 
	\begin{align*}
	e_s(a^{(n)},P_{\theta,\mu})  \leq & \left(\frac{2}{C_1}\right)^{\frac{1}{d+r}} (n-2)^{-\frac{1}{d+r}} e_r\left(a^{ \left\lceil \frac{n}{2} \right\rceil },P\right)^{\frac{r}{d+r}} \\
	\leq & \, \frac{2^{1+\frac{1}{d}}C_0^2\,r^{\frac 1d}\,}{d^{\frac1d}M^{\frac1d}V_d^{\frac1d}\mathfrak{p}(A,|\cdot-a_1|)^{\frac1d}}\min_{\varepsilon \in (0,\frac13)}\big[\varphi_r(\varepsilon)^{-\frac{1}{d}}\big] \theta^{\frac ds} \left(\int_{\{f>0\}} f(x)^{-|q|(1+a)} f_{\theta,\mu}^{\frac{d+r}{d+r-s}}(x) d\lambda_d(x) \right)^{\frac{1}{|q|(d+r)}}\\
	& \times \|f\|_{\frac{d}{d+r}}^{\frac{1}{d+r}}\|f\|_a^{\frac{a}{d+r}}(n-2)^{-\frac{1}{d}}.
	\end{align*}
	The result is deduced using the same arguments as in the end of the proof of theorem $\ref{thmgreedygeneral1}$. 
	\hfill $\square$
\end{refproofhybrid}
\subsection{Example of distributions with finite polynomial moments up to a finite order}
Theorem $\ref{Pierce1}$ treats the case of a distribution $P$ that has finite polynomial moments at any order. However, this condition is not always satisfied. The goal of this example is to see what happens if the distribution $P$ has finite moments up to a finite order $r+\delta$ i.e when there exists a finite number $M$ such that $\E|X|^{r+\delta}<+\infty$ for $r+\delta<M$. For this, let us consider the hyper-Cauchy distribution $P=f.\lambda_d$ where 
$$f(x)=\frac{Cm}{(1+|x|^2)^m}$$
for a finite constant $C>0$ and $m>\frac d2$, this ensures the integrability of $f$ w.r.t. the Lebesgue measure $\lambda_d$. This probability distribution has finite moments of order $r+\delta<2m-d$, i.e. $\E|X|^{r+\delta}<+\infty$ if $r+\delta<2m-d$.
\\
In order to obtain Pierce type results, one proceeds as in the proof of Theorem \ref{Pierce1}.  Criterion $(\ref{criterenu})$ is verified with 
$$g_{\ve}(x)=\frac{K_{\delta,r}}{\big(1\vee (1+\ve)|x-a_1|\big)^{d\frac{r+\delta}{r+\delta-\eta}}}$$
and the reasoning is the same until inequality $(\ref{holderbigintegral})$.  At this stage, since $P$ does not have finite moments of any order, one wonders if the above inequality makes sense, i.e. if the integrals in the right side are finite. First, it is clear that 
\begin{equation}
\label{condint}
\int_{\{f>0\}}\left(\frac{f_{\theta,\mu}}{f}\right)^{(1-q)q'}fd\lambda_d =\int \left( \frac{1+|x|^2}{1+\theta^2 |x|^2}\right)^{m(1-q)q'}\frac{Cm}{(1+|x|^2)^m}d\lambda_d<+\infty
\end{equation}
where $1-q=\frac{d+r}{d+r-s}$ and $q',p'$ are two conjugate coefficients larger than $1$, 
since $\frac{1+|x|^2}{1+\theta^2 |x|^2}$ is bounded for $\theta>0$ and  $\frac{Cm}{(1+|x|^2)^m}\in L^1(\lambda_d)$ as mentioned previously.
Secondly, one notices that, since $\frac{r+\delta}{r+\delta-\eta}>1$, then $|q|p'd\frac{r+\delta}{r+\delta-\eta}=|q|dp'+\eta'$ for some $\eta'>0$. Hence, one can write
$$\int_{\R^d}g_{\ve}^{qp'}dP<+\infty \quad \Leftrightarrow \quad \int_{\R^d}\frac{ |x|^{|q|dp'+\eta'}}{(1+|x|^2)^m}d\lambda_d(x)<+\infty   \quad \Leftrightarrow \quad   \int_{0}^{+\infty}\frac{ |y|^{|q|dp'+\eta'+d-1}}{(1+|y|^2)^m}dy<+\infty.$$
This is equivalent to 
$$2m-\big(d|q|p'+\eta'+d-1\big)>1 \quad \Leftrightarrow \quad p'<\frac{1}{d|q|}(2m-d-\eta')<\frac{2m-d}{d|q|}.$$
At this stage, we note that one can choose $p'$ as close to $1$ as possible, since its conjugate $q'$ can be chosen as large as possible without affecting  $(\ref{condint})$. Hence, the above condition 
boils down to 
$d|q|<2m-d \quad \Leftrightarrow \quad \frac{s}{d+r-s}<\frac{2m-d}{d}$. Consequently, in order for this study to have sense, one must have 
$$s<\Big(1-\frac{d}{2m}\Big)(d+r)$$
which is more restrictive than the condition $s<d+r$ in the case of distributions with finite moments of any order.
 	\section{Upper estimates for $L^r$-optimal quantizers}
 	\label{optimaldilat}
Let $r,s>0$ and $(\Gamma^n)_{n \geq 1}$ a sequence of $L^r(\R^d)$-optimal quantizers of a random vector $X$ with probability distribution $P$. For every $\mu \in \R^d$ and $\theta>0$, we denote $\Gamma^n_{\theta,\mu}=\mu+\theta(\Gamma^n-\mu)=\{\mu+\theta(x_i-\mu), \; x_i\in \Gamma^n, \, 1\leq i \leq n\}$. \\

In $\cite{Sagna08}$, the $L^s$-optimality of the sequence $(\Gamma^{n}_{\theta,\mu})_{n \geq 1}$ was studied. The author provided some conditions for the $L^s$-rate optimality of this sequence depending on whether $\Gamma^n$ is an asymptotically $L^r$-optimal quantizer (study done for $s<r$) or exactly $L^r$-optimal (for $s<r+d$). This study was based on the integrability of the $b$-maximal functions associated to an $L^r$-optimal sequence of quantizers $(\Gamma^n)_{n \geq 1}$ defined by 
\begin{equation}
\label{psib}
\forall \xi \in \R^d, \quad \Psi_b(\xi)=\sup_{n \in \N}\frac{\lambda_d\left(B(\xi,b \, \mbox{dist}(\xi,\Gamma^n))\right)}{P\left(B(\xi,b \, \mbox{dist}(\xi,\Gamma^n))\right)}.
\end{equation}
Throughout this section, we focus on the case where $\Gamma^n$ is exactly $L^r$-optimal and  $0<s<r+d$ and extend the results established in \cite{Sagna08} to a larger class of distributions using tools that appeared meanwhile in $\cite{papiergreedy}$. Instead of maximal functions, our study relies on micro-macro inequalities using auxiliary probability distributions $\nu$ satisfying the following control on balls with respect to an $a_1 \in \Gamma^n$: for every $\varepsilon \in (0,1)$, there exists a Borel function $g_{\varepsilon}: \R^d \rightarrow (0,+\infty)$ such that, for every $x \in \mbox{supp}(\P)$ and every $t \in [0,\varepsilon |x-a_1|]$, 
\begin{equation}
\label{criterenuopt}
\nu(B(x,t)) \geq g_{\varepsilon}(x) V_d t^d.
\end{equation}
where $V_d$ denotes the volume of the hyper unit ball.  
\subsection{Main results}
The case where $r<s$ and $(\Gamma^n)_{n \geq 0}$ is a sequence of $L^r$-asymptotically optimal quantizers of $P$ has been studied in $\cite{Sagna08}$ without the use of maximal functions but requiring the couple $(\theta,\mu)$ to be $P$-admissible, i.e. such that 
$$\{f>0\} \subset \mu (1-\theta)+\theta\{f>0\}.$$
Note that if $\mbox{supp}(P)=\R^d$, then every couple $(\theta,\mu)$ is $P$-admissible. This condition is not needed to establish upper error bounds in this paper but will be considered in the studies for $s<r$ in Section \ref{examplesdilat}.
\begin{thm}
	\label{specificoptimal}
	Let $s\in [r,d+r)$ and $1-q=\frac{d+r}{d+r-s}$. Let $X$ be an $\R^d$-valued random vector with distribution $P=f.\lambda_d$ such that $\E|X|^{r+\delta}<+\infty$ for some $\delta>0$ such that $r+\delta>\frac{sd}{d+r-s}$. Let $\eta \in \big(0, r+\delta-\frac{sd}{d+r-s}\big)$, $p'=\frac{r+\delta-\eta}{d|q|}$ and $q'=\frac{r+\delta-\eta}{r+\delta-\eta-d|q|}$ and  let $(\Gamma^n)_{n \geq 1}$ be a sequence of $L^r(\R^d)$-optimal quantizers of $X$. Assume 
$$\int_{\{f>0\}}\left(\frac{f_{\theta,\mu}}{f}\right)^{(1-q)q'}f d\lambda_d< +\infty.$$
Then, for every $n \geq 1$, \\
$$e_s(\Gamma^n_{\theta,\mu},P)  \leq  \widetilde \kappa^{\text{Optimal}}_{\theta,\mu} \theta^{1+\frac {d}{s}} \sigma_{r+\delta}(P)\left( \int_{\{f>0\}}\left(\frac{f_{\theta,\mu}}{f}\right)^{(1-q)q'}fd\lambda_d \right)^{\frac{1}{|q|q'(d+r)}}n^{-\frac1d}$$
	where $\sigma_{r+\delta}(P)=\inf_{a}\|X-a\|_{r+\delta}$ is the $L^{r+\delta}$-standard deviation of $P$ and
	\begin{align*}
	\widetilde{\kappa}^{\text{Optimal}}_{\theta,\mu} &=2^{\frac{2qp'-1}{qp'(d+r)}}\left(\frac{(2^r-1)C_1^r+2^rC_2^r}{V_d} \right)^{\frac{1}{d+r}} \,  \min_{\varepsilon \in (0,\tfrac13)}\Big[(1+\ve)\varphi_r(\varepsilon)^{-\frac{1}{d+r}}\Big]
	\Big( \int_{\R^d} \left(1 \vee |x|\right)^{-d\frac{r+\delta}{r+\delta-\eta}} dx \Big)^{\frac{1}{d+r}}
	 \end{align*}
	  with $C_1$ and $C_2$ are finite constants not depending on $n, \theta$ and $\mu$ and $\varphi_r:u \rightarrow \left(\frac{1}{3^r}-u^r\right)u^d, \; u \in (0,\tfrac13)$.
\end{thm}
\begin{rmq}
One checks that $\varphi_r$ attains its maximum at $\frac 13 \left( \frac{d}{d+r}\right)^{\frac 1r}$ on $(0,\tfrac13)$. 
\end{rmq}
Note that, like for greedy quantizeration sequences, the case $s<r$ can be easily treated by remarking that $e_s(\Gamma_{\theta,\mu}^n,P) \leq e_r(\Gamma^n_{\theta,\mu},P)$ which is upper bounded in Theorem \ref{specificoptimal}.
 	\subsection{Proof}
 	We start with a general theoretical result based on the auxiliary distribution $\nu$ and its companion function $g_{\ve}$ satisfying $(\ref{criterenuopt})$.
\begin{thm}
\label{upperoptimal}
	Let $s\in (0,d+r)$ and $1-q=\frac{d+r}{d+r-s}$. Let $X$ be an $\R^d$-valued random vector with distribution $P=f.\lambda_d$ such that $\E|X|^{r+\delta}<+\infty$ for some $\delta>0$ and let $(\Gamma^n)_{n \geq 1}$ be a sequence of $L^r(\R^d)$-optimal quantizers of $X$ such that $\Gamma^n=\{x_1,\ldots,x_n\}$. Assume there exist a distribution $\nu$ and a function $g_{\ve}$ satisfying $(\ref{criterenuopt})$, for $\ve\in (0,\tfrac13)$, such that
$$ \int_{\{f>0\}} \left( \frac{f_{\theta,\mu}}{f \, g_{\ve}}\right)^{|q|} dP_{\theta,\mu}< +\infty.$$
Then, for every $n \geq 1$, \\
$$e_s(\Gamma^n_{\theta,\mu},P)  \leq  \kappa^{\text{Optimal}}_{\theta,\mu} \theta^{1+\frac {d}{d+r}} \left(\int_{\{f>0\}}  \left( \frac{f_{\theta,\mu}}{f \, g_{\ve}}\right)^{|q|}(x) dP_{\theta,\mu}(x) \right)^{\frac{1}{|q|(d+r)}} n^{-\frac1d}$$
where $\kappa^{\text{Optimal}}_{\theta,\mu}=\Big(4(2^r-1)C_1^r+4.2^rC_2^r \Big)^{\frac{1}{d+r}} V_d^{-\frac{1}{d+r}} \,  \min_{\varepsilon \in (0,\tfrac13)}\big[\varphi_r(\varepsilon)^{-\frac{1}{d+r}}\big]$ with $C_1$ and $C_2$ finite constants not depending on $n, \theta$ and $\mu$ and $\varphi_r:u \rightarrow \left(\frac{1}{3^r}-u^r\right)u^d, \; u \in (0,\tfrac13)$.
\end{thm}
\begin{proof}
	First, as in the proof of Theorem $\ref{thmgreedygeneral1}$, we have for every $n\geq 1$,\begin{align}
	\label{linkthetamu}
	e_s(\Gamma^n_{\theta,\mu},P)^s= \theta^s e_s(\Gamma^n,P_{\theta,\mu})^s.
	\end{align} 
	Then, assume that $c \in (0 , \frac{\varepsilon}{1-\varepsilon}] \cap (0,\frac12)$ so that $\frac{c}{c+1} \leq \varepsilon$. Moreover, $d(x,\Gamma^n)\leq |x-a_1|$ for an $a_1 \in \Gamma^n$. So, $\frac{c}{c+1}d(x,\Gamma^n) \leq \varepsilon |x-a_1|$ and, hence, $\nu$ satisfies $(\ref{criterenuopt})$ w.r.t. $a_1$. Consequently, there exists a Borel function $g_{\varepsilon}:\R^d \rightarrow (0,+\infty)$ such that 
	$$\nu \left(B\left(x, \frac{c}{c+1} \,d\big(x,\Gamma^n\big)  \right) \right) \geq V_d \, \left(\frac{c}{c+1} \right)^{d} d(x,\Gamma^n)^d\, g_{\varepsilon}(x).$$
	Then, noticing that $ \frac{(1-c)^r-c^r}{(1+c)^r} \geq \frac{1}{3^r} -\big( \frac{c}{c+1}\big)^r \; >0$ in $(\ref{micromacrooptimal})$, since $c \in (0,\tfrac 12)$, yields
	\begin{equation*}
	e_r(\Gamma^n,P)^r - e_r(\Gamma^{n+1},P)^r  
	\geq V_d \, \varphi_r\left(\frac{c}{c+1} \right) \int g_{\varepsilon}(x) d(x,\Gamma^n)^{d+r} dP(x)
	\end{equation*}  
	where $\ds \varphi_r(u)=\left(\frac{1}{3^r}-u^r\right)u^d, \; u \in (0,\tfrac13)$. This inequality is the version of $(\ref{equ1})$ for optimal quantizers so we follow the same steps as in the proof of Theorem \ref{thmgreedygeneral1} until we obtain
	\begin{equation*}
	e_r(\Gamma^n,P)^r - e_r(\Gamma^{n+1},P)^r  \geq C e_s(\Gamma^n,P_{\theta,\mu})^{d+r}
	\end{equation*}
	where $\ds C=V_d \, \varphi_r\left(\frac{c}{c+1} \right) \theta^{-d} \left(\int_{\{f>0\}}  \left( \frac{f_{\theta,\mu}}{f \, g_{\ve}}\right)^{|q|} dP_{\theta,\mu}(x) \right)^{\frac1q}.$ 	At this stage, since $(\Gamma^n)_{n\geq 1}$ is a sequence of $L^r$-optimal quantizers, we use $(\ref{increment})$ to obtain the following upper bound
\begin{align}
	\label{avecphi}
	e_s(\Gamma^n,P_{\theta,\mu})  \leq & C^{-\frac{1}{d+r}} \left( \frac{4(2^r-1)e_r(\Gamma^{n+1},P)^r}{n+1}+\frac{4.2^rC_2^r n^{-\frac rd}}{n+1}\right)^{\frac{1}{d+r}} \nonumber \\
	\leq & \left( \frac{4(2^r-1)e_r(\Gamma^{n+1},P)^r}{n+1}+\frac{4.2^rC_2^r n^{-\frac rd}}{n+1}\right)^{\frac{1}{d+r}}V_d^{-\frac{1}{d+r}} \, \varphi_r\left(\frac{c}{c+1} \right)^{-\frac{1}{d+r}} \theta^{\frac {d}{d+r}} \nonumber \\
	&\times \left(\int_{\{f>0\}}  \left( \frac{f_{\theta,\mu}}{f \, g_{\ve}}\right)^{|q|} dP_{\theta,\mu}(x) \right)^{\frac{1}{|q|(d+r)}}\nonumber \\
	\leq &\, \Big(4(2^r-1)C_1^r+4.2^rC_2^r \Big)^{\frac{1}{d+r}} n^{-\frac1d} V_d^{-\frac{1}{d+r}} \, \varphi_r\left(\frac{c}{c+1} \right)^{-\frac{1}{d+r}} \theta^{\frac {d}{d+r}}\nonumber \\
	& \times \left(\int_{\{f>0\}}  \left( \frac{f_{\theta,\mu}}{f \, g_{\ve}}\right)^{|q|} dP_{\theta,\mu}(x) \right)^{\frac{1}{|q|(d+r)}}
	\end{align}
	where we used, in the last inequality, the definition of an $L^r$-optimal quantizer given by $(\ref{defrateoptimal})$.
	Now, we use $(\ref{varphimin})$ to obtain 
	\begin{align*}
	e_s(\Gamma^n,P_{\theta,\mu})  \leq &\,  \Big(4(2^r-1)C_1^r+4.2^rC_2^r \Big)^{\frac{1}{d+r}} n^{-\frac1d}V_d^{-\frac{1}{d+r}} \,  \theta^{\frac {d}{d+r}}\, \min_{\varepsilon \in (0,\tfrac13)}\big[\varphi_r(\varepsilon)^{-\frac{1}{d+r}}\big]\\
	& \times \left(\int_{\{f>0\}}  \left( \frac{f_{\theta,\mu}}{f \, g_{\ve}}\right)^{|q|} dP_{\theta,\mu}(x) \right)^{\frac{1}{|q|(d+r)}}.
	\end{align*}
	Finally, one deduces the result by injecting this last inequality in $(\ref{linkthetamu})$.
	\hfill $\square$\\
\end{proof}
 	By specifying the function $g_{\varepsilon}$ in Theorem $\ref{upperoptimal}$, we obtain  a universal non asymptotic bound for the error $e_s(\Gamma^n_{\theta,\mu},P)$ given in Theorem \ref{specificoptimal} which proof is the following.\medskip\\
\begin{refproofoptspecific}
	We consider $\nu(dx)=\gamma_{r,\delta}(x) \lambda_d(dx)$ where
	$$\gamma_{r,\delta}(x)=\frac{K_{\delta,r}}{(1 \vee |x-a_1|)^{d\frac{r+\delta}{r+\delta-\eta}}}\;  \qquad \mbox{with} \qquad  K_{\delta,r}= \left(\int \frac{dx}{(1 \vee |x|)^{d\frac{r+\delta}{r+\delta-\eta}}} \right)^{-1} <+\infty$$
	is a probability density with respect to the Lebesgue measure on $\R^d$ and $|\cdot|$ denotes any norm on $\R^d$. Similarly as in the proof of Theorem $\ref{Pierce1}$, $(\ref{criterenuopt})$ is verified with $$g_{\varepsilon}(x)=\frac{K_{\delta,r}}{\big(1 \vee(1+\varepsilon)|x-a_1|\big)^{d\frac{r+\delta}{r+\delta-\eta}}}.$$
	So, we apply Theorem $\ref{upperoptimal}$  and use $(\ref{holderbigintegral})$ to obtain
	$$e_s(\Gamma^n_{\theta,\mu},P)\leq  \kappa^{\text{Optimal}}_{\theta,\mu} \theta^{1+\frac {d}{s}} K_{\delta,r}^{-\frac{1}{d+r}} \left(\int_{\R^d} g_{\ve}^{qp'}dP\right)^{\frac{1}{|q|p'(d+r)}} \left( \int_{\{f>0\}}\left(\frac{f_{\theta,\mu}}{f}\right)^{(1-q)q'}fd\lambda_d \right)^{\frac{1}{|q|q'(d+r)}}n^{-\frac1d}$$	
	where $q=\frac{-s}{d+r-s}$ so that $1-q=\frac{d+r}{d+r-s}$ and $p'$ and $q'$ are two conjugate coefficients larger than $1$.
	By our choice of $g_{\ve}$ and the $L^{r+\delta}$ Minkowski inequality, 
	\begin{align*}
	\left(\int_{\R^d}g_{\ve}^{qp'}dP\right)^{\frac{1}{|q|p'(d+r)}} & \leq K_{\delta,r}^{-\frac{1}{d+r}} \Big(1+(1+\ve)\sigma_{r+\delta}(P)\Big)^{\frac{1}{|q|p'(d+r)}}.
	\end{align*}
	where 
	$\sigma_{r+\delta}(P)=\inf_{a}\|X-a\|_{r+\delta}$ is the $L^{r+\delta}$-standard deviation of $P$. Consequently, one has
	\begin{align*}
	e_s(\Gamma^n_{\theta,\mu},P)  \leq  \kappa^{\text{Optimal}}_{\theta,\mu} \theta^{1+\frac {d}{s}} K_{\delta,r}^{-\frac{1}{d+r}} \Big(1+(1+\ve)\sigma_{r+\delta}(P)\Big)^{\frac{1}{|q|p'(d+r)}} \left( \int_{\{f>0\}}\left(\frac{f_{\theta,\mu}}{f}\right)^{(1-q)q'}fd\lambda_d \right)^{\frac{1}{|q|q'(d+r)}}n^{-\frac1d}
	\end{align*}
	Now, we introduce an equivariance argument. For $\lambda>0$, let $X_{\lambda}:=\lambda(X-a_1)+a_1$ and $(\alpha_{\lambda,n})_{n\geq 1}:=(\lambda(\alpha_n-a_1)+a_1)_{n \geq 1}$. It is clear that  $e_r(\alpha^{(n)},X)=\frac{1}{\lambda}e_r(\alpha_{\lambda}^{(n)},X_{\lambda})$. 
		Plugging this in the previous inequality yields
		\begin{align*}
	e_s(\Gamma^n_{\theta,\mu},P)  \leq  \frac{\kappa^{\text{Optimal}}_{\theta,\mu}}{K_{\delta,r}^{\frac{1}{d+r}}} \theta^{1+\frac {d}{s}} \frac{1}{\lambda}\Big(1+(1+\ve)\lambda \sigma_{r+\delta}(P)\Big)^{\frac{1}{|q|p'(d+r)}} \left( \int_{\{f>0\}}\left(\frac{f_{\theta,\mu}}{f}\right)^{(1-q)q'}fd\lambda_d \right)^{\frac{1}{|q|q'(d+r)}}n^{-\frac1d}
	\end{align*}
	Finally, one deduces the result by setting $\lambda=\frac{1}{(1+\ve)\sigma_{r+\delta}(P)}$.	
	\hfill $\square$
\end{refproofoptspecific}
 	\section{More examples and a dilatation optimization}
\label{examplesdilat}
Let $X$ be a random variable with distribution $P=f.\lambda_d$. The upper bounds established in Sections $\ref{greedydilat}$ and $\ref{optimaldilat}$, induce that the quantizers $\Gamma^n_{\theta,\mu}$ and $a^{(n)}_{\theta,\mu}$ are $L^s(P)$-rate optimal under one of the following necessary and sufficient conditions depending on the value of $s$, as follows \\
$\rhd$ If $s<r$ and  $(\theta,\mu)$ is $P$-admissible, then $a^{(n)}_{\theta,\mu}$ is  $L^s(P)$-rate optimal iff $P$ has finite moments of order $r+\delta$ for $\delta>0$ and
\begin{equation}
\label{cond1}
\int f^{-\frac{s}{r-s}}f_{\theta,\mu}^{\frac{r}{r-s}}d\lambda_d <+\infty.
\end{equation}
Note that it is the same condition for $\Gamma^n_{\theta,\mu}$ but this case is fully treated in $\cite{Sagna08}$.\\
$\rhd$ If $s<r+d$, then the $L^r$-dilated greedy sequence $a^{(n)}_{\theta,\mu}$ and the $L^r$-dilated optimal sequence $\Gamma^n_{\theta,\mu}$ are $L^s(P)$-rate optimal iff $P$ has finite moments of order $r+\delta$ for $\delta>0$ and
\begin{equation}
\label{cond2}
 \int_{\{f>0\}} \left( \frac{f_{\theta,\mu}}{f}\right)^{\frac{(d+r)(r+\delta-\eta)}{(d+r-s)(r+\delta-\eta)-ds}}fd\lambda_d <+\infty
\end{equation}
where $\eta \in \big(0, r+\delta-\frac{sd}{d+r-s}\big)$. 
In particular, when $f$ is a radial non-increasing density,  the $L^r$-dilated greedy sequence $a^{(n)}_{\theta,\mu}$ is $L^s(P)$-rate optimal iff $P$ has finite moments of order $\frac{1-a}{a}(d+\ve)$, $\ve>0$, and
\begin{equation}
\label{condrad}
\int f(x)^{\frac{-s(1+a)}{d+r-s}} f_{\theta,\mu}^{\frac{d+r}{d+r-s}}(x) d\lambda_d(x) <+\infty
\end{equation}
where $a\in (0,1)$.\\

This leads to determining the values of $(\theta,\mu)$ for which these conditions are satisfied and hence obtain an interval $I_P(\theta,\mu)$ of the parameters for which the $L^r$-dilated sequence is $L^s$-optimal. Let us denote, for the sake of simplicity, $\alpha^{(n)}_{\theta,\mu}$ both sequences $(\Gamma^n_{\theta,\mu})_{n \geq 1}$ and $(a^{(n)}_{\theta,\mu})_{n \geq 1}$. Generally, $\mu$ is chosen to be equal to $\E [X]$ in order to ensure that the distribution $P_{\theta,\mu}$ lies in the same family of distributions of $P$, and the values of $\theta$ for which the above conditions are satisfied depend entirely on the density $f$ of $P$. So, the problem is to determine the interval $I_P(\theta)$ depending on the distribution $P$. This way, based on $L^r$-optimal or greedy sequences $\alpha^{(n)}$, we obtain sequences $\alpha^{(n)}_{\theta,\mu}$ that are $L^s$-rate optimal, but not optimal nor even $L^s$-asymptotically optimal. We will carry out the study for specified families of distributions, like the multivariate Normal distribution $\mathcal{N}(m,\Sigma)$, the hyper-exponential, hyper-Gamma and hyper-Cauchy distributions. For each case, we determine the interval $I_P(\theta)$ and show that the dilated/contracted sequence does not satisfy the $L^s$-empirical measure theorem for every $\theta \in I_P(\theta)$. However, the computations established  allow us to determine, for some probability distributions, a particular value $\theta^*\in I_P(\theta)$ for which the sequence $\alpha_{\theta^*,\mu}^{(n)}$ satisfies the theorem. Let us first recall this theorem.
\begin{thm}[Empirical measure theorem]
	\label{empiricalmeasure}
	Let $P$ be a $L^r$-Zador distribution, absolutely continuous w.r.t the Lebesgue measure on $\R^d$ with density $f$. Let $\Gamma^n$ be an asymptotically optimal $n$-quantizer of $P$. Then, denoting $\ds C_{f,r}=\int_{\R^d} f^{\frac{d}{d+r}} d\lambda_d $, one has
	\begin{equation}
	\frac{1}{n} \sum_{x_i \in \Gamma^n} \delta_{x_i} \underset{n\rightarrow +\infty}{\Rightarrow} P_r=\frac{1}{C_{f,r}}\int f^{\frac{d}{r+d}}d\lambda_d,
	\end{equation}
	or, in other words, for every $a,b \in \R^d$, 
	$$\frac1n \,{\rm card}\, \big\{x_i \in \Gamma^n \cap [a,b] \big\} \rightarrow \frac{1}{C_{f,r}}\int_{[a,b]} f^{\frac{d}{r+d}}d\lambda_d.$$
\end{thm}
Moreover, for some distributions, the particular value $\theta^*$ mentioned above allows the lower bound $(\ref{lowerbound})$ induced by $\alpha_{\theta^*,\mu}^{(n)}$ to attain the sharp constant in Zador's theorem. This leads to wonder whether this sequence is $L^s$-asymptotically optimal.\\

Before proceeding with the particular studies, let us precise that, if $\theta>1$, the sequence $\alpha^{(n)}_{\theta,\mu}$ is called a dilatation of $\alpha^{(n)}$ with scaling parameter $\theta$ and translating number $\mu$. Likewise,  if $\theta<1$, the sequence $\alpha^{(n)}_{\theta,\mu}$ is called a contraction of $\alpha^{(n)}$ with scaling parameter $\theta$ and translating number $\mu$.
\subsection{The multivariate Gaussian distribution}
Let $P=\mathcal{N}(m, \Sigma)$. We consider $\mu=m$ so that the distribution $P_{\theta,\mu}$ lies in the same family of distributions as $P$. Since $\mbox{supp}(P)=\R^d$, then every couple $(\theta,\mu)$ is $P$-admissible.\\ 
$\rhd$ If $s<r$, the sequence $\alpha^n_{\theta,m}$ is $L^s$-rate optimal iff $\theta \in I_P(\theta)=\big(\sqrt{\frac sr}, +\infty \big)$. These computations are carried out in $\cite{Sagna08}$ for optimal quantizers and are the same for greedy quantizers.\\
$\rhd$ If $r \leq s <d+r$, we lead two studies, relying first on condition $(\ref{cond2})$ and then on condition $(\ref{condrad})$for radial densities and see what link we can make between both of them. Let us start with the general case, i.e. condition $(\ref{cond2})$. For $q=\frac{-s}{d+r-s}$ and every $q'>1$, one has 
$$\int _{\{f>0\}} \left(\frac{f_{\theta,m}}{f}\right)^{(1-q)q'}fd\lambda= \big((2\pi)^d|\Sigma|\big)^{-\frac12} \int e^{-\frac12\big((1-q)q'\theta^2+(q-1)q'+1\big)(x-m)^2|\Sigma|^{-2}}dx.$$
So, the sequence $\alpha_{\theta,m}^{(n)}$ is $L^s$-rate optimal iff
$$(1-q)q'\theta^2+(q-1)q'+1>0 \quad \Leftrightarrow \quad \theta^2> 1-\frac{1}{q'(1-q)}$$
and this for every $q'>1$. So, one can consider $q'$ as close to $1$ as possible and deduce that  $(\ref{cond2})$ is satisfied iff
$$I_P(\theta)=\Big(\sqrt{\frac{s}{d+r}},+\infty \Big).$$
Now, since the Normal distribution is a radial density distribution, it is interesting to see what the condition $(\ref{condrad})$ yields. For every $a\in (0,1)$, one has 
$$\int_{\{f>0\}} f^{q(1-a)}f_{\theta,m}^{1-q}d\lambda_d= \big((2\pi)^d|\Sigma| \big)^{-\frac12} \int e^{-\frac12\big(q(1+a)+\theta^2(1-q) \big) (x-m)^2|\Sigma|^{-2}}d\lambda_d.$$
So, $\alpha_{\theta,m}^{(n)}$ is $L^s$-rate optimal iff
$$(1-q)\theta^2+q(1+a)>0 \quad \Leftrightarrow \quad \theta^2>\frac{s}{d+r}(1+a)$$
and this for every $a\in(0,1)$. At this stage, note that the Normal distribution has finite $r$-th moment for every $r>0$ so the moment assumption made in Theorem $\ref{zadorpierce}$ allows us to choose $a$ as small as possible in a way that, even if $\frac{(1-a)(d+\ve)}{a}$ goes to infinity, we can still apply the theorem. Hence, one chooses $a \rightarrow 0^+$ and the condition made on $\theta$ reads $\theta^2>\frac{s}{d+r}$ and the interval $I_P(\theta)$ becomes 
 $$I_P(\theta)=\Big(\sqrt{\frac{s}{d+r}},+\infty \Big)$$
 coinciding with the interval deduced from condition $(\ref{cond2})$ as explained in Remark \ref{remsura}.
 \begin{rmq}
 One should note that choosing a scalar $\theta^*$ is optimal in the case of radial density probability distributions but, in the general case, it would be more precise if $\theta^*$ is a matrix.
 \end{rmq}
\paragraph{Empirical measure theorem}
This study relies on the fact that the $L^r$-quantizers themselves satisfy the $L^r$-empirical measure theorem so it is conducted only for $L^r$-dilated optimal quantizers $\Gamma^n_{\theta,m}$ since greedy quantizers do not satisfy this theorem. In order to conclude whether the sequence $\Gamma_{\theta,m}^{n}$ satisfies the empirical measure theorem, we start by determining the ``{\em limit measure}'' of the empirical measure, i.e. determine the limit of $\frac1n \,{\rm card}\, \big\{x_i \in \Gamma^n_{\theta,m} \cap [a,b] \big\}$. For every $n\geq 1$, it is clear that 
	$$\big\{x_i \in \Gamma^n_{\theta,m} \cap [a,b] \big\} = \Big\{x_i \in \Gamma^n \cap \Big[\frac{a}{\theta},\, \frac{b}{\theta}\Big] \Big\}.$$
	So, since $\Gamma^n$ satisfies the $L^r$-empirical measure theorem, 
	\begin{align*}
	\frac1n \mbox{card} \{x_i \in \Gamma^n_{\theta,m} \cap [a,b] \}& \rightarrow \frac{1}{C_{f,r}}\int_{\big[\frac{a}{\theta},\, \frac{b}{\theta}\big]} f^{\frac{d}{r+d}}d\lambda_d = \frac{1}{C_{f,r}}\theta^{-d}\int_{[a,b]} f\Big(\frac{x-m}{\theta}+m\Big)^{\frac{d}{r+d}}d\lambda_d
	\end{align*}
	where $C_{f,r}=\int_{\R^d} f^{\frac{d}{d+r}}d\lambda_d$. For every $\theta \in I_P(\theta)$, one has
	$$\int_{[a,b]}f\Big(\frac{x-m}{\theta}+m\Big)^{\frac{d}{r+d}}d\lambda_d= \big((2\pi)^d|\Sigma|\big)^{\frac{-d}{2(d+r)}}\int_{[a,b]} e^{-\frac12\frac{d}{d+r}\theta^{-2}(x-m)^2|\Sigma|^{-2}}d\lambda_d$$
	and
	$$\int_{\R^d} f^{\frac{d}{d+r}}d\lambda_d= \big((2\pi)^d|\Sigma|\big)^{\frac{r}{2(d+r)}}\left(\frac{d+r}{d}\right)^{\frac d2}.$$
	So, the limit of the empirical measure is given by
	\begin{align*}
	\frac1n \mbox{card} \{x_i \in \Gamma^n_{\theta,m} \cap [a,b] \}& \rightarrow \Big(\frac{d+r}{d\theta^2}\Big)^{\frac d2}  \big((2\pi)^d|\Sigma|\big)^{-\frac12+\frac12 \frac{d}{(d+r)\theta^2}} \int_{[a,b]} f^{\frac{d}{(d+r)\theta^2}}d\lambda_d\\
	& = \frac{1}{\int_{\R^d}f^{\frac{d}{(d+r)\theta^2}}d\lambda_d}\int_{[a,b]}f^{\frac{d}{(d+r)\theta^2}}d\lambda_d.
	\end{align*}
	
	With this limit, one clearly does not find the limit needed to satisfy the empirical measure theorem for every $\theta \in I_P(\theta)$. Instead, one can notice that it is possible for a particular value $\theta^*$ given by 
	$$\frac{d}{d+r}{\theta^*}^{-2}=\frac{d}{d+s} \qquad \Leftrightarrow \qquad \theta^*=\sqrt{\frac{d+s}{d+r}}.$$
	This leads to the following Proposition.
	\begin{prop}
	Let $r,s >0$ and $P=\mathcal{N}(m,\Sigma)$ be a multivariate Normal distribution. Assume $\Gamma^n$ is an asymptotically $L^r$-optimal quantizer of $P$. Consider
	$$\theta^*=\sqrt{\frac{d+s}{d+r}},$$ then the sequence $\Gamma^n_{\theta^*,m}$ satisfies the $L^s$-empirical measure theorem, i.e. 
	$$\frac1n\, {\rm card}\, \big\{x_i \in \Gamma^n_{\theta^*,m} \cap [a,b] \big\} \rightarrow \frac{1}{C_{f,s}}\int_{[a,b]} f^{\frac{d}{s+d}}d\lambda_d.$$
\end{prop}
	This has been shown in $\cite{Sagna08}$ in addition to the fact that this particular $\theta^*$ minimizes the upper bound of the $L^s$-quantization error $e_s(\Gamma_{\theta,\mu}^{(n)},P)$ induced by the $L^r$-dilated optimal quantizer of the Normal distribution. Moreover, the author has showed that, even if the lower bound $(\ref{lowerbound})$ coincide with the sharp limiting constant in  Zador's Theorem for this value of $\theta^*$, the sequence $\Gamma^{(n)}_{\theta^*,m}$ is still not $L^s$-asymptotically optimal. 
\subsection{Hyper-exponential distributions}
\label{ex1}
Let $X \sim P=f.\lambda_d$ where 
$f(x)=e^{-\lambda |x|^{\alpha}}$
for $\alpha, \lambda >0$ and $|.|$ denotes a norm on $\R^d$. We consider $\mu=0$ so that the distribution $P_{\theta,\mu}$ lies in the same family of distributions as $P$. Note that if one considers the density function $f(x)=e^{-\lambda |x-m|^{\alpha}}$ for $m \in \R$, the study will be the same since the quantities considered are invariant by translation. In other words, if $\Gamma$ is an optimal quantizer of $X$, then $\Gamma -m(1, \ldots,1)$ is an optimal quantizer for $X-m$. Moreover, it is clear that every couple $(\theta,\mu)$ is $P$-admissible.\\
	$\rhd$ If $s<r$, one has
	\begin{align*}
	\int f^{-\frac{s}{r-s}}(x)f_{\theta,0}^{\frac{r}{r-s}}(x)dx= \int e^{-\frac{s\lambda |x|^{\alpha} }{s-r}} e^{-\frac{r\lambda|\theta x|^{\alpha} }{r-s}} = \int e^{-\lambda \left(\frac{s}{s-r}+\frac{r}{r-s} \theta^{\alpha}\right)|x|^{\alpha}}
	\end{align*}
	So $\alpha^{n}_{\theta,0}$ is $L^s$-optimal iff $(\ref{cond1})$ is satisfied which is clearly equivalent to $\theta^{\alpha}> \frac{s}{r}.$ Hence, the interval $I_P(\theta)$ is equal to
	$$I_P(\theta)= \left(\left( \frac{s}{r}\right)^{\frac{1}{\alpha}}, +\infty \right).$$ 
	$\rhd$ For $s\in(r,d+r)$, the idea is as follows. Just as for the Normal distribution, the hyper-Exponential distribution has finite moments of order $r$ for every $r>0$ so the moment assumption made in Theorem $\ref{zadorpierce}$ allows us to choose $a$ as small as possible and the condition $(\ref{condrad})$ coincides with condition $(\ref{cond2})$ as explained in Remark \ref{remsura}. Consequently, we will lead the study relying on $(\ref{cond2})$. One has, for $q=\frac{-s}{d+r-s}$ and every $q'>1$, that
	\begin{align*}
	\int \left( \frac{f_{\theta,0}}{f}\right)^{(1-q)q'}fd\lambda_d = \int e^{-\lambda\big((1-q)q'\theta^{\alpha}+(q-1)q'+1\big)|x|^{\alpha}}d\lambda_d.
	\end{align*}
	So $\alpha^{n}_{\theta,0}$ is $L^s$-optimal iff $(\ref{cond2})$ is satisfied which is clearly equivalent to 
	$$(1-q)q'\theta^{\alpha}+(q-1)q'+1>0 \quad \Leftrightarrow \quad \theta^{\alpha}> 1-\frac{1}{q'(1-q)}$$
and this for every $q'>1$. Hence, one can choose $q'$ as small as possible, for example $q'\rightarrow1^+$, yielding
$$I_P(\theta)=\left(\Big(\frac{s}{d+r}\Big)^{\frac{1}{\alpha}},+\infty\right).$$
\paragraph{Empirical measure theorem}
As explained in the previous example, this study is conducted for $L^r$-dilated optimal quantizers. As previously, we start by determining  the limit of the empirical measure
\begin{align*}
	\frac1n \mbox{card} \{x_i \in \Gamma^n_{\theta} \cap [a,b] \}& \rightarrow \frac{1}{C_{f,r}}\int_{\big[\frac{a}{\theta},\, \frac{b}{\theta}\big]} f^{\frac{d}{r+d}}d\lambda_d = \frac{1}{C_{f,r}}\theta^{-d}\int_{[a,b]} f\Big(\frac{x}{\theta}\Big)^{\frac{d}{r+d}}d\lambda_d
	\end{align*}
	where $C_{f,r}=\int_{\R^d} f^{\frac{d}{d+r}}d\lambda_d$. For every $\theta \in I_P(\theta)$, 
	$$\int_{[a,b]} f(\theta^{-1}x)^{\frac{d}{r+d}}d\lambda_d= \int_{[a,b]}e^{-\lambda\frac{d}{d+r}\theta^{-\alpha}|x|^{\alpha}}=\int_{[a,b]}f(x)^{\frac{d}{(d+r)\theta^{\alpha}}}d\lambda_d.$$
Moreover, one uses the fact that \begin{equation}
\label{outilintexp}
\int_{\R^d}f(|x|)dx=V_d \int_{0}^{+\infty} f(r) r^{d-1}dr \qquad \mbox{and} \qquad  \int_{0}^{+\infty} x^ne^{-ax^b}dx=\frac{\Gamma\big(\frac{n+1}{b}\big)}{ba^{(n+1)/b}},
\end{equation} where $V_d=V(B_d)$ is the volume of the hyper-unit ball on $\R^d$ and $\Gamma$ is the Gamma function, to obtain
	$$\int_{\R^d}f^{\frac{d}{d+r}}d\lambda_d= \int_{\R^d} e^{-\lambda \frac{d}{d+r}|x|^{\alpha}}d\lambda_d=V_d\frac{\Gamma(\frac{d}{\alpha})}{\alpha} \left(\lambda\frac{d}{d+r} \right)^{-\frac{d}{\alpha}}.$$
	By the same arguments, one deduces that $$\int_{\R^d}f(x)^{\frac{d}{(d+r)\theta^{\alpha}}}d\lambda=\frac{1}{\theta^dC_{f,r}}$$
	so that the limiting measure is 
		\begin{align*}
	\frac1n \mbox{card} \{x_i \in \Gamma^n_{\theta,m} \cap [a,b] \}& \rightarrow  \frac{1}{\int_{\R^d}f^{\frac{d}{(d+r)\theta^{\alpha}}}d\lambda_d}\int_{[a,b]}f^{\frac{d}{(d+r)\theta^{\alpha}}}d\lambda_d.
	\end{align*}
	Consequently, we deduce that the sequence $\Gamma^{(n)}_{\theta,0}$ does not satisfy the empirical measure theorem for every $\theta \in I_P(\theta)$ except for a particular value $\theta^*$ given by 
	$$\frac{d}{d+r} {\theta^*}^{-\alpha}=\frac{d}{d+s} \qquad \Leftrightarrow \qquad \theta^*=\left(\frac{d+s}{d+r}\right)^{\frac{1}{\alpha}}$$
	hence leading to the following Proposition
	\begin{prop}
	Let $r,s >0$ and $P=f.\lambda_d$ where 
	$f(x)=e^{-\lambda |x|^{\alpha}}$
	for $\alpha, \lambda >0$. Assume $\Gamma^n$ is an asymptotically $L^r$-optimal quantizer of $P$. Consider
	$$\theta^*=\left(\frac{d+s}{d+r}\right)^{\frac{1}{\alpha}},$$ then the sequence $\Gamma^n_{\theta^*,0}$ satisfies the $L^s$-empirical measure theorem, i.e. 
	$$\frac1n\, {\rm card}\, \big\{x_i \in \Gamma^n_{\theta^*,0} \cap [a,b] \big\} \rightarrow \frac{1}{C_{f,s}}\int_{[a,b]} f^{\frac{d}{s+d}}d\lambda_d.$$
\end{prop}
\noindent
Note that $\theta^*$ does not depend on the parameter $\lambda$ of the distribution, only on $\alpha$.
In the next proposition, we show that the sequence $\alpha^n_{\theta^*,0}$ satisfies the lower bound $(\ref{lowerbound})$.
	\begin{prop}
		Let $r,s >0$ and $P=f.\lambda_d$ where 
		$f(x)=e^{-\lambda |x|^{\alpha}}$
		for $\alpha, \lambda >0$. Then, the asymptotic lower bound of the $L^s$-error of the sequence $\alpha^n_{\theta^*,0}$ with $\theta^*=\left(\frac{d+s}{d+r}\right)^{\frac{1}{\alpha}}$ satisfies 
		$$Q_{r,s}^{\rm Inf}(P,\theta^*)=Q_s(P)$$
		where $\ds Q_{r,s}^{\rm Inf}(P,\theta^*)=(\theta^*)^{s+d} \widetilde{J}_{s,d} \left( \int f^{\frac{d}{d+r}}d\lambda_d \right)^{\frac sd} \int f^{-\frac{s}{d+r}}(x)f_{\theta^*,0}(x)dx$. 
	\end{prop}
\begin{proof}
	Elementary computations based on $(\ref{outilintexp})$ show that
	$$\int f^{-\frac{s}{d+r}}(x)f_{\theta^*,0}(x)dx =V_d \frac{\Gamma(\frac{d}{\alpha})}{\alpha \lambda^{-\frac{d}{\alpha}}} \left(\frac{d}{r+d} \right)^{-\frac {d}{\alpha}} \qquad 
	\mbox{and} \qquad \int f^{\frac{d}{d+r}}d\lambda_d=V_d \frac{\Gamma(\frac{d}{\alpha})}{\alpha \lambda^{-\frac{d}{\alpha}}}\left(\frac{d}{r+d} \right)^{-\frac {d}{\alpha}} $$
	so that 
	$$(\theta^*)^{s+d} \left( \int f^{\frac{d}{d+r}}d\lambda_d \right)^{\frac sd} \int f^{-\frac{s}{d+r}}(x)f_{\theta^*,0}(x)dx =\left(V_d \frac{\Gamma(\frac{d}{\alpha})}{\alpha}\lambda^{-\frac{d}{\alpha}} \right)^{1+\frac ds} \left( \frac{s+d}{d}\right)^{\frac{s+d}{\alpha}}=\left( \int f^{\frac{d}{d+s}}d\lambda_d \right)^{\frac {d+s}{d}} $$
	and hence the result.
	\hfill $\square$
\end{proof}
\medskip

It is interesting to see whether $\Gamma^{(n)}_{\theta^*,0}$ for $\theta^{*}=\left(\frac{s+d}{r+d} \right)^{\frac{1}{\alpha}}$ is $L^s$-asymptotically optimal. For this, we compute the upper bound of the $L^s$-quantization error $e_s(\Gamma_{\theta^*,0}^n,P)$ given in  in Corollary $\ref{specificoptimal}$ and see if it reaches the sharp constant in Zador's Theorem for the different values of $s$. Note that if $\alpha^{(n)}$ is a greedy quantization sequence, one cannot make any interesting conclusions since it is clear that the sharp Zador constant cannot be attained by our upper bounds. \\
Let $r,s >0$ and $\Gamma^n$ an $L^r$-optimal quantizer of $P$. Elementary computations based on $(\ref{outilintexp})$ show that the upper bounds of the quantization error of $P$ induced by $\Gamma^n_{\theta^*,0}$, for $\theta^*=\left(\frac{s+d}{r+d}\right)^{\frac{1}{\alpha}}$, are given by
\[
Q_{r,s}^{\sup,\theta^*}=
\left\{
\begin{array}{lr}
\widetilde{J}_{r,d}^{\frac1r} \left(\int f^{\frac{d}{d+s}d\lambda_d} \right)^{\frac{d+s}{ds}} & \mbox{if } s<r,\\
 \widetilde{\kappa}_{\theta^*,m}^{\text{Optimal}}  \left(\frac{V_d \Gamma(\frac{d}{\alpha})}{\alpha \lambda^{\frac{d}{\alpha}}} \right)^{\frac1s} \left(\frac{s+d}{d}\right)^{\frac{d}{s\alpha}} \left(\frac{s+d}{r+d} \right)^{\frac{1}{\alpha}}  & \mbox{if } r<s<d+r.
\end{array}
\right.
\]
One can easily notice that, for the different values of $s$, $Q_s(P) \leq Q_{r,s}^{\sup,\theta^*}$. Consequently, no conclusions can be made on the $L^s$-asymptotically optimality of the sequence $(\Gamma^n_{\theta^*,0})_{n \geq 0}$. However, if we have $\widetilde{J}_{s,d}^{\frac1s}$ instead of $\widetilde{J}_{r,d}^{\frac1r}$, then one can reach Zador's sharp constant for $r<s$ and gets closer to it for $s \in (r,d+r)$.
\subsection{Hyper-Gamma distributions}
\label{ex2}
Let $X \sim P=f.\lambda_d$ where 
$f(x)=|x|^{\beta}e^{-\lambda |x|^{\alpha}}$
for $\alpha, \lambda >0$ and $\beta>-d$ and $|\cdot|$ denotes any norm on $\R^d$. We consider $\mu=0$ so that $P_{\theta,\mu}$ lies in the same family of distributions as $P$. In this case, every couple $(\theta,\mu)$ is $P$-admissible since $\mbox{supp}(P)=\R^d$.\\
$\rhd$ If $s<r$, one has
	\begin{align*}
	\int f^{-\frac{s}{r-s}}(x)f_{\theta,0}^{\frac{r}{r-s}}(x)dx= \theta^{\frac{r\beta}{r-s}}\int |x|^{\beta} e^{-\lambda \left(\frac{s}{s-r}+\frac{r}{r-s} \theta^{\alpha}\right)|x|^{\alpha}}
	\end{align*}
	So $\alpha^{n}_{\theta,0}$ is $L^s$-optimal iff $(\ref{cond1})$ is satisfied which is clearly equivalent to 
	$\theta^{\alpha}> \frac{s}{r}.$ Consequently,
	$$I_P(\theta)=\left(\left(\frac sr\right)^{\frac{1}{\alpha}},+\infty\right).$$
	$\rhd$ If $s<d+r$, the conditions $(\ref{cond2})$ and $(\ref{condrad})$ yield the same result as explained in Remark $\ref{remsura}$. For $q=\frac{-s}{d+r-s}$ and every $q'>1$, one has
	\begin{align*}
	\int\left( \frac{f_{\theta,0}}{f}\right)^{(1-q)q'}f(x)d\lambda_d = \int|x|^{\beta} e^{-\lambda\big((1-q)q'\theta^{\alpha}+(q-1)q'+1\big)|x|^{\alpha}}d\lambda_d.
	\end{align*}
	So $\alpha^{n}_{\theta,0}$ is $L^s$-optimal iff 
	$$(1-q)q'\theta^{\alpha}+(q-1)q'+1>0 \quad \Leftrightarrow \quad \theta^{\alpha}> 1-\frac{1}{q'(1-q)}$$
and this for every $q'>1$. Hence, one can choose $q'$ as small as possible, for example $q'\rightarrow1^+$, yielding
$$I_P(\theta)=\left(\Big(\frac{s}{d+r}\Big)^{\frac{1}{\alpha}},+\infty\right).$$
\paragraph{Empirical measure theorem}
As explained in the previous examples, this study is conducted for $L^r$-dilated optimal quantizers. 
First, we compute the limit
\begin{align*}
	\frac1n \mbox{card} \{x_i \in \Gamma^n_{\theta,0} \cap [a,b] \}& \rightarrow \frac{1}{C_{f,r}}\int_{\big[\frac{a}{\theta},\, \frac{b}{\theta}\big]} f^{\frac{d}{r+d}}d\lambda_d = \frac{1}{C_{f,r}}\theta^{-d}\int_{[a,b]} f\Big(\frac{x}{\theta}\Big)^{\frac{d}{r+d}}d\lambda_d
	\end{align*}
	where $C_{f,r}=\int_{\R^d} f^{\frac{d}{d+r}}d\lambda_d$. For every $\theta \in I_P(\theta)$, 
	$$\int_{[a,b]} f(\theta^{-1}x)^{\frac{d}{r+d}}d\lambda_d =\theta^{-\frac{d\beta}{d+r}}\int_{[a,b]}|x|^{\frac{d\beta}{d+r}}e^{-\lambda\frac{d}{d+r}\frac{1}{\theta^{\alpha}}|x|^{\alpha}}d\lambda_d.$$
Moreover, using $(\ref{outilintexp})$ yields
	$$\int_{\R^d}f^{\frac{d}{d+r}}d\lambda_d= \int_{\R^d} |x|^{\frac{d\beta}{d+r}}e^{-\lambda \frac{d}{d+r}|x|^{\alpha}}d\lambda_d=V_d\frac{\Gamma(\frac{d+\frac{d\beta}{d+r}}{\alpha})}{\alpha} \left(\lambda\frac{d}{d+r} \right)^{-\frac{1}{\alpha}\Big(d+\frac{d\beta}{d+R}\Big)}.$$
	Likewise, one obtains
	 $$\int_{\R^d}|x|^{\frac{\beta d (\theta^{\alpha}-1)}{\theta^{\alpha}(d+r)}}f(x)^{\frac{d}{(d+r)\theta^{\alpha}}}d\lambda=C_{f,r}\theta^{d+\frac{d\beta}{d+r}}.$$
	Consequently, the limiting measure is 
		\begin{align*}
	\frac1n \mbox{card} \{x_i \in \Gamma^n_{\theta,m} \cap [a,b] \}& \rightarrow  \frac{1}{\int_{\R^d}|x|^{\frac{\beta d (\theta^{\alpha}-1)}{\theta^{\alpha}(d+r)}}f^{\frac{d}{(d+r)\theta^{\alpha}}}d\lambda_d}\int_{[a,b]}|x|^{\frac{\beta d (\theta^{\alpha}-1)}{\theta^{\alpha}(d+r)}}f^{\frac{d}{(d+r)\theta^{\alpha}}}d\lambda_d.
	\end{align*}
	Hence, in order for the sequence $\Gamma^{(n)}_{\theta,0}$ to satisfy the empirical measure theorem, there is two conditions to fulfill 
	$$\frac{d}{(d+r)\theta^{\alpha}}=\frac{d}{d+s} \qquad \mbox{and} \qquad \frac{\beta d (\theta^{\alpha}-1)}{\theta^{\alpha}(d+r)}=0.$$ 
	This is true for $$\beta^*=\frac{d+r}{d(d+s)} \qquad \mbox{and} \qquad  \theta^{*}=\left(\frac{d+s}{d+r}\right)^{\frac{1}{\alpha}}.$$
So, one can deduce with the following proposition.
\begin{prop}
	Let $r,s >0$ and $P=f.\lambda_d$ where 
	$f(x)=|x|^{\beta}e^{-\lambda |x|^{\alpha}}$
	for $\alpha, \lambda >0$ and $\beta>-d$ and $|\cdot|$ is any norm on $\R^d$. Assume $\Gamma^n$ is an asymptotically $L^r$-optimal quantizer of $P$. Consider
	$$\beta=\frac{d+r}{d(d+s)}\qquad \mbox{and} \qquad \theta^*=\left(\frac{d+s}{d+r}\right)^{\frac{1}{\alpha}},$$ then the sequence $\Gamma^n_{\theta^*,0}$ satisfies the $L^s$-empirical measure theorem, i.e. 
	$$\frac1n\, {\rm card}\, \big\{x_i \in \Gamma^n_{\theta^*,0} \cap [a,b] \big\} \rightarrow \frac{1}{C_{f,s}}\int_{[a,b]} f^{\frac{d}{s+d}}d\lambda_d.$$
\end{prop}
Note that one obtains the same results  for the distribution with density $|x-m|^{\beta}e^{-\lambda |x-m|^{\alpha}}$ since it is invariant by translation. \\

Elementary computations, similar to those established previously, show that one cannot make any conclusions on the $L^s$-optimality of the $L^r$-dilated sequence considering the values of $\beta$ and $\theta^*$ deduced in the previous proposition. In other words, one cannot know whether the lower and upper bound of the $L^s$-quantization error induced by $\alpha^n_{\theta^*,0}$ are equal or comparable to the sharp limiting constant $Q_s(P)$ in Zador's Theorem. 
\subsection{Numerical observations}
We just showed that, for a particular value $\theta^*$, the sequence $\alpha^{(n)}_{\theta^*,\mu}$ satisfies the $L^s$-empirical measure theorem and that the lower bound of the $L^s$-quantization error induced by this sequence attains the sharp constant in Zador's Theorem, the upper bound only getting close. This pushes to conjecture that the optimally $L^r$-dilated sequence $(\alpha^n_{\theta^*,\mu})$  is asymptotically $L^s$-optimal. Numerical experiments were established in \cite{Sagna08} to prove this conjecture numerically for optimal quantizers. In this section, we implement similar experiments to come to this type of conclusion for optimally $L^r$-dilated greedy quantization sequences. We denote $a^{r,(n)}$ the $L^r$-greedy quantization sequence.  
\paragraph{Normal distribution}
We start with the Normal distribution $\mathcal{N}(0,1)$ and compute the corresponding $L^3$-optimal greedy quantization sequence $a^{3,(n)}$ by a standard Newton Raphson algorithm on one hand, and the optimally $L^2$-dilated greedy quantization sequence $a^{2,(n)}_{\theta^*,\mu}$ with $\theta^*=\sqrt{\frac{s+d}{r+d}}=\sqrt{\frac43}$ and $\mu=0$, on the other hand. We make a linear regression of the two resulting sequences for different values of the size $n$ and expose, in Table $\ref{reg1}$, the corresponding regression coefficients. 
\begin{table}
\begin{center}
\begin{tabular}{cc|cc|cc}
\hline
\multicolumn{2}{c|}{Normal Distribution} & \multicolumn{2}{c|}{Exponential distribution} & \multicolumn{2}{c}{$P=f.\lambda_d$ with $f(x)=x^2 e^{-x^2}$} \\
\hline
$n$ & Regression coefficient & $n$ & Regression coefficient & $n$ & Regression coefficient\\
\hline
$255$ & $0.9818$ & $ 373$ & $ 0.981$ & $255$ & $0.9399$\\
$511$ & $0.9855$ & $745$ & $0.988$ & $511$ & $0.9405$\\
 $1\,023$ & $0.9945$ & $1\, 489$ & $0.990$ & $1\, 023$ & $0.9406$\\
 \hline
\end{tabular}
\end{center}
\vspace{-0.5cm}
\caption{Regression coefficients of the optimally $L^2$-dilated greedy sequence on the $L^3$-optimal greedy sequence for $\mathcal{N}(0,1)$, $\mathcal{E}(1)$ and $P=f.\lambda_d$ with $f(x)=x^2e^{-x^2}$.}
\label{reg1}
\end{table}
%
%
\paragraph{Exponential distribution}
We consider the exponential distribution $\mathcal{E}(1)$ with parameter $\lambda=1$. In other words, it is the distribution studied in Example \ref{ex1} for $d=1$ and $\alpha=1$. We compute the $L^3$-optimal greedy quantization sequence $a^{3,(n)}$ by a Newton Raphson algorithm and the optimally $L^2$-dilated greedy quantization sequence $a^{2,(n)}_{\theta^*,\mu}$ with $\theta^*=\left(\frac{s+d}{r+d}\right)^{\frac{1}{\alpha}}=\frac43$ and $\mu=0$. The $L^2$-optimal greedy quantization sequence is obtained by a standard Lloyd's algorithm. We expose, in Table \ref{reg1}, the regression coefficients obtained by regressing the $L^2$-dilated sequences on the $L^3$ greedy sequences.
\paragraph{Hyper-Gamma distribution}
 Let $d=1$. We consider the Hyper-Gamma probability distribution with parameters $\lambda=1$ and $\alpha=\beta=2$ so the density is given  by $$f(x)=x^2e^{-x^2}.$$ 
In example $\ref{ex2}$, we showed that the hyper-Gamma distribution satisfy the $L^s$-empirical measure for a particular parameter $\beta$ and a particular $\theta^*\in I_P(\theta)$. However, we conduct here the experiment for different values and see if one always have the same convergence of the regression coefficients to $1$. We compute the $L^3$-optimal greedy quantization sequence $a^{3,(n)}$ by a Newton Raphson algorithm and the $L^2$-optimal greedy quantization sequence $a^{2,(n)}$ by a Lloyd's algorithm. The optimally $L^2$-dilated greedy sequence is given by $\alpha^{2,(n)}_{\theta^*,\mu}$ with $\theta^*=\left( \frac{s+d}{r+d}\right)^{\frac{1}{\alpha}}=\sqrt{\frac43}$ and $\mu=0$. Table \ref{reg1} shows the regression coefficients obtained by regressing the $L^2$-dilated sequences on the $L^3$ greedy sequences where we observe a slower convergence, even a divergence of the coefficients to $1$, hence deducing that this sequence cannot be $L^s$-asymptotically optimal.
\paragraph{Conjecture}
For the Normal and exponential distributions, the regression coefficient converges to $1$ for specific values of $n$. This leads us to conjecture that there exists a sub-sequence of the greedy quantization sequence for which the regression coefficient converges to $1$, i.e. for which the sequence is asymptotically $L^s$-optimal.\\
 In fact, this ``{\em subsequence}'' topic has already been investigated in $\cite{papiergreedy}$  where it has been shown (numerically) that there exist sub-optimal greedy quantization sequences, in the sense that the graphs representing the weights of the Vorono\"i cells converge towards the density curve of the distribution for certain sizes $n$ of the sequence. For example, the greedy quantization sequence of $\mathcal{N}(0,1)$, and more generally of symmetrical distributions around $0$, is sub-optimal and the optimal sub-sequence is of the form $a^{(n)}=a^{(2^k-1)}$ for $k \in \N^*$.\\
 Hence, it is natural to conjecture that the optimally $L^r$-dilated sub-sequences of the same size are asymptotically $L^s$-optimal. 
\section{Application to numerical integration}
\label{application}
Optimal quantizers and greedy quantization sequences are used in numerical probability where one relies on cubature formulas  to approximate the exact value of $\E f(X)$, for a continuous bounded function $f$ and a random variable $X$ with distribution $P$, by
\begin{equation}
\label{quad}
\E f(X) \approx \E f(\widehat X^{\alpha^{(n)}}) = \sum_{i=1}^{n} p_i^n f(\alpha_i^n)
\end{equation}
where $\alpha^{(n)}$ designates the optimal or greedy quantization sequence of the random variable $X$ and $p_i^n=P\big(X \in W_i(\alpha^{(n)})\big)$ represents the weight of the $i^{\rm th}$ Vorono\"i cell corresponding to $\alpha^{(n)}$ for every $i \in \{1, \ldots,n\}$. A new iterative formula for the approximation of $\E f (X)$  using greedy quantization sequences is given in \cite{papiergreedy}, based on the recursive character of greedy quantization. Upper error bounds of these approximations have been investigated repeatedly in the literature, in \cite{papiergreedy, Pages15, Pages18} for example. \\

In this section, we present what advantages the dilated quantization sequences bring to the numerical integration field. This application was first introduced  in $\cite{Sagna08}$ by A. Sagna for optimal quantizers. Here, we briefly recall his idea and emphasize that it also works with dilated greedy quantization sequences as well.\\

Let $X \in L^{\beta}, \beta \in (2,+\infty)$ and let $f$ be a locally Lipschitz function, in the sense that, there exists a bounded constant $C>0$ such that 
\begin{equation}
\label{loclip}
|f(x)-f(y)|\leq C |x-y|\big(1+|x|^{\beta-1}+|y|^{\beta-1}\big).
\end{equation}
For every quantizer $\alpha^{(n)}$ (not necessarily stationary), one has, by applying H\"older's inequality with the conjugate exponents $r$ and $r'=\frac{r}{r-1}$, that 
\begin{align}
\label{ubound}
\big| \E f(X)-\E f(\widehat X^{\alpha^{(n)}})\big| \leq  \E\big|f(X)-f(\widehat X^{\alpha^{(n)}})\big| & \leq C \,\E\Big( \big|X-\widehat X^{\alpha^{(n)}}\big|\, \big(1+|X|^{\beta-1}+|\widehat X^{\alpha^{(n)}}|^{\beta-1}\big)\Big)\nonumber \\
 &\leq  C\, \big\|X-\widehat X^{\alpha^{(n)}}\big\|_r\, \Big(1+\|X\|_{(\beta-1)r'}^{\beta-1} +\|\widehat X^{\alpha^{(n)}}\|_{(\beta-1)r'}^{\beta-1} \Big).
\end{align}
In order for this upper bound to make sense, one should have 
\begin{equation}
\label{condr}(\beta-1)r'=\frac{(\beta-1)r}{r-1}\leq \beta \qquad \Longleftrightarrow \qquad r \geq \beta>2.
\end{equation}

In practice, since most algorithms to optimize quantization (of $n$-tuples of greedy sequences) are much easier to implement in the quadratic case, it is more convenient to use such quadratic optimal or greedy quantizers in this type of applications to approximate expectations of the form $\E f(X)$. However, if we use $L^2$-quantizers $\alpha^{(n)}$ in our case, we obtain upper bounds involving an $L^r$-quantization error for $r>2$ (see $(\ref{condr})$) which is not really optimal since the quantizer used is not $L^r$-optimal for $r>2$. So, an idea is to use $L^2$-dilated quantizers $\alpha^{(n)}_{\theta,\mu}$ which is itself $L^r$-rate optimal for given values of $\theta$ and $\mu$ depending on the probability distribution $P$. For example, if $X\sim \mathcal{N}(m,I_d)$, then one chooses $\mu=m$ and $\theta=\sqrt{\frac{r+d}{2+d}}$. \\

Hence, one approximates $\E f(X)$ by $\E f(\widehat X^{\alpha^{(n)}_{\theta,\mu}})$ rather than $\E f(\widehat X^{\alpha^{(n)}})$ via
$$\E f(\widehat X^{\alpha^{(n)}_{\theta,\mu}})=\sum_{i=1}^{n} p_i^{\theta,\mu} f(\alpha_i^{\theta,\mu})$$
with $p_i^{\theta,\mu}$ being the weight of the $i^{th}$ Vorono\"i cell corresponding to the quantization sequence $\alpha^{(n)}_{\theta^*,\mu}$ given by 
\begin{equation}
\label{dilatedweight}
P\big(X \in W_i(\alpha^{(n)}_{\theta^*,\mu}) \big)=\int _{W_i(\alpha^{(n)}_{\theta^*,\mu})} f(x) d\lambda_d(x)= \theta^d \int _{W_i(\alpha^{(n)})}f_{\theta^*,\mu} (z) d\lambda_d(z)
=P\big(\widehat X^{\alpha^{(n)}_{\theta^*,\mu}} \in W_i(\alpha^{(n)}) \big)
\end{equation}
where we applied the change of variables $z=\mu+\frac{x-\mu}{\theta}$. Then, since $\|X-\widehat X^{\alpha^{(n)}_{\theta,\mu}}\|_r$ converges faster to $0$ than $\|X-\widehat X^{\alpha^{(n)}}\|_r$ for $r>2$ if we consider an $L^2$-quantizer $\alpha^{(n)}$, one may expect to observe that
$$\big| \E f(X)-\E f(\widehat X^{\alpha^{(n)}_{\theta,\mu}})\big| \leq \big| \E f(X)-\E f(\widehat X^{\alpha^{(n)}})\big|.$$

To illustrate this numerically, we consider a one-dimensional example and approximate $\E f(X)$, where $X$ is a random variable with Normal distribution $\mathcal{N}(0,1)$ and $f$ is defined on $\R$ by $f(x)=x^4+\sin(x)$ and satisfies $(\ref{loclip})$ with $\beta=5$. To satisfy $(\ref{condr})$, we choose $r=5$ and implement the approximation by quadrature formulas based, on the one hand, on $L^2$-optimal and greedy sequences $\alpha^{(n)}$ and, on the other hand, on the $L^2$-dilated optimal and greedy quantizer $\alpha^{(n)}_{\theta^*,0}$, with $\theta^*=\sqrt{\frac{r+d}{2+d}}=\sqrt{2}$, which is $L^r$-rate optimal. The exact value of $\E f(X)$ is $3$. In figure \ref{ex1dilat}, we illustrate the errors induced by these approximations and we observe that, for a same size $n$ of the quantization sequence, the $L^2$-dilated quantizers $\alpha^{(n)}_{\theta^*,0}$ give more precise results than the standard sequences $\alpha^{(n)}$ themselves. \\

\begin{figure}
\begin{center}
\begin{tabular}{cc}
			\includegraphics[width=8cm,height=7cm,angle=0]{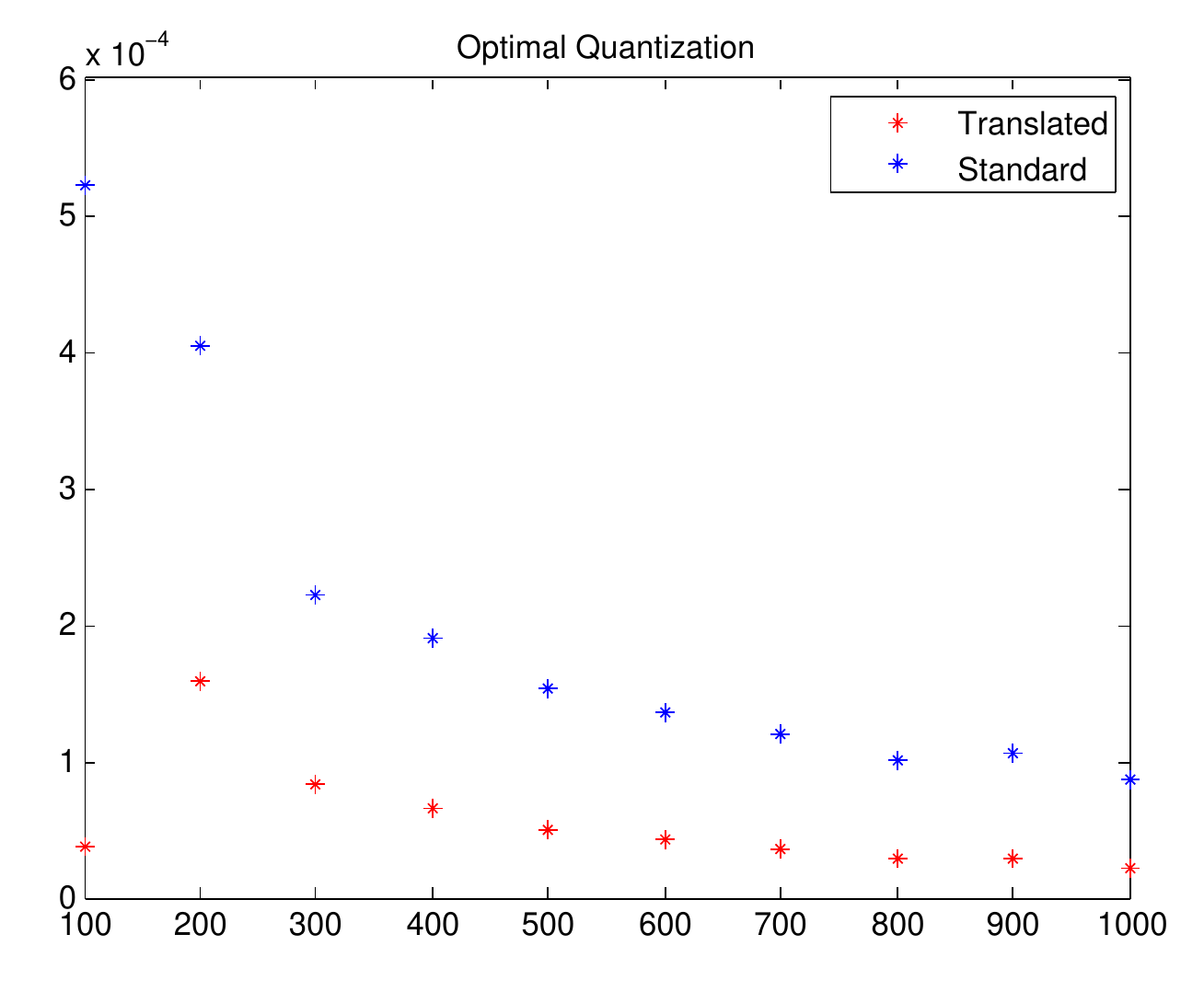} &
			\includegraphics[width=8cm,height=7cm,angle=0]{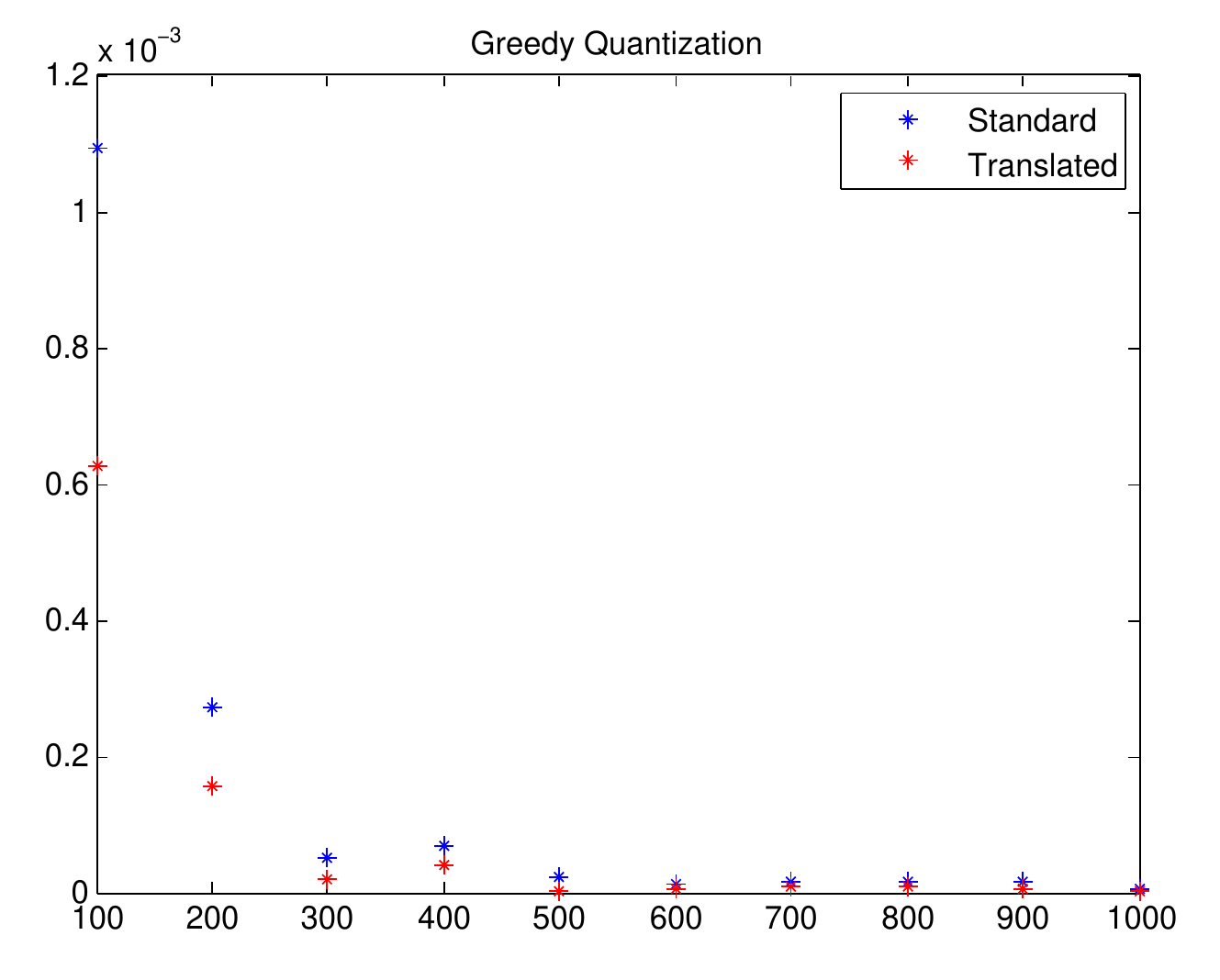} 
			\end{tabular}
		\end{center} 
		\vspace{-0.5cm}
		\caption{Errors of the approximation of $\E f(X)$, where $f(x)=x^{4}+\sin(x)$, by quadrature formulas based on $L^2$ quantizers (blue) and dilated $L^2$ quantizers (red) for different sizes $n$.}
		\label{ex1dilat}
\end{figure}
\noindent {\em Acknowledgments.} I would like to express a sincere gratitude to my supervisor, Pr. Gilles Pag\`es, for his help and advice during this work and to Dr. Rami El Haddad, my co-supervisor, as well. Also, I would like to acknowledge the National Council for Scientific Research of Lebanon (CNRS-L) for granting me a doctoral fellowship, in a joint program with Agence Universitaire de la Francophonie of the Middle East and the research council of Saint-Joseph University of Beirut.
\small

	\end{document}